\documentclass[11pt]{article}%
\usepackage{amssymb,amsmath,amsfonts,amsthm,array,bm,bbm,color}%
\usepackage{graphicx}
\providecommand{\U}[1]{\protect\rule{.1in}{.1in}}

\numberwithin{equation}{section}

\newtheorem{theorem}{Theorem}[section]
\newtheorem{lemma}[theorem]{Lemma}

\newtheorem{proposition}[theorem]{Proposition}
\newtheorem{remark}[theorem]{Remark}

\newtheorem{definition}[theorem]{Definition}
\newtheorem{hypothesis}[theorem]{Hypothesis}

\def\<{\langle}
\def\>{\rangle}
\def\E{\mathbb{E}}
\def\P{\mathbb{P}}
\def\R{\mathbb{R}}
\def\T{\mathbb{T}}
\def\Z{\mathbb{Z}}

\begin{document}

\title{On the Boussinesq hypothesis for a stochastic\\ Proudman-Taylor model}

\author{Franco Flandoli\footnote{Email: franco.flandoli@sns.it. Scuola Normale Superiore of Pisa, Piazza dei Cavalieri 7, 56124 Pisa, Italy}\qquad
Dejun Luo\footnote{Email: luodj@amss.ac.cn. Key Laboratory of RCSDS, Academy of Mathematics and Systems Science, Chinese Academy of Sciences, Beijing 100190, China and School of Mathematical Sciences, University of Chinese Academy of Sciences, Beijing 100049, China}}
\maketitle
\vspace{-15pt}

\begin{abstract}
We introduce a stochastic version of Proudman-Taylor model, a 2D-3C fluid
approximation of the 3D Navier-Stokes equations, with the small-scale
turbulence modeled by a transport-stretching noise. For this model we may
rigorously take a scaling limit leading to a deterministic model with
additional viscosity on large scales. In certain choice of noises without mirror
symmetry, we identify an AKA effect. This is the first example with a 3D structure
and a stretching noise term.
\end{abstract}



\section{Introduction}

\subsection{General comments}

The purpose of this paper is proving that a certain model of a 3D fluid,
incorporating small-scale turbulence by means of a suitable stochastic
modification, manifests the property of dissipation on large scales predicted
by Joseph Boussinesq in 1877 (cf. \cite{Bous}). See Section
\ref{subsect def Boussinesq} below for an illustration of the
property in the framework of stochastic 3D Navier-Stokes equations in vorticity form; the latter will be heuristically derived in Section \ref{sect heuristics} as a starting point of investigation of the Boussinesq property. This property is at the foundation of Large Eddy Simulation (LES) models but its validity is
critical \cite{JiangLayton, Schmi}. And its proof under suitable circumstances
is a question which may count a number of contributions but not a definite
answer, see for instance \cite{Berselli, JiangLayton, LabLay, Wirth}.

From the numerical or experimental side, it is perhaps true that this property
is better verified in certain 3D cases. In 2D, inverse cascade seems to
deteriorate the structure of small-scale turbulence and scale separation
needed to have dissipation at large scales, so the result is true but for a
limited time \cite{Cummins}, and after a while perturbations appear which even
lead to a phenomenon often quoted as negative viscosity. In 3D, small-scale
turbulence is more persistent and thus better oriented to the validity of the
property investigated here; however, other elements in 3D may work against
(for instance the direct cascade) and the picture is perhaps more difficult
than in 2D, see e.g. \cite{Wirth}.

Concerning simplified models in 2D with stochastic inputs as those discussed in
\cite{FGL21a, FGL21c, FLL23, Luo21a}, it has been proved that an additional
positive turbulent viscosity emerges (the validity of such result compared to
direct numerical simulations \cite{Cummins} is limited to an initial time
range). Similar unpublished computations, however, seem to indicate that the
property fails for some 3D fluid models, due to a lack of control on the
stochastic stretching term, as discussed in Section
\ref{subsect def Boussinesq} (see also \cite{FL, Luo23} which involve only the
transport noise). It emerges, among the possible ideas, the one that only a
suitable geometry of the involved vector fields could remedy the excess of
fluctuations in the stretching term. In particular, if the large-scale
vorticity field and the small-scale turbulent velocity field would be
orthogonal, this would cancel the above mentioned stretching contribution. But
such geometry is too extreme and essentially boils down to the 2D problem.

Here we investigate a model which is in between the 2D and the 3D ones, the
so-called 2D-3C model (called also 2.5D model \cite{Seshasa}), where the
dependence on space variables is 2D but the vector fields have three
components. The Physics behind this simplification has been identified by
Proudman \cite{Proud} and Taylor \cite{Taylor} and corresponds to situations
with large Coriolis or rotation forces \cite{Babin, BMN99}. In the transient, corrections to the
2D-3C model should be included, but they will be the
object of future investigation. We think that the phenomena appearing in the
simple 2D-3C model are interesting enough to justify a separate study.

The model is formulated below in Section \ref{2D 3C reduction}. The small
scales are modeled by a noise and in a suitable scaling limit we prove the
existence of the additional dissipation at large scales predicted by
Boussinesq \cite{Bous}.

The noise, as opposed to most previous works \cite{FGL21a, FGL21c, Galeati20}
where it was defined by Fourier decomposition, is modeled on idealized but
still meaningful 3D vortex structures, following \cite{FH}. A Fourier
description is possible and illustrated in Section \ref{subsec-series-noise}, but the assumptions on the covariance
function are suggested by the vortex structure interpretation. In Section
\ref{sect heuristics} we explain the heuristic behind this choice. This new
noise has been essential for us to identify the presence of the so-called AKA
(anisotropic kinetic alpha) effect \cite{Wirth}. It is the presence of a term
with first order derivatives, in addition to the dissipative second order term
predicted by Boussinesq. As shown below and in analogy with previous studies
based on different models and techniques like \cite{Wirth}, the AKA\ effect is
absent when mirror-symmetry (namely parity invariance) is imposed on the
small-scale turbulence. Removing mirror symmetry, still we may have or not
have the AKA effect (as in \cite{Wirth}); we provide examples of both cases.

In the rest of the introduction, we first describe, via the idea of separation of scales, the origin of noise in the vorticity formulation of 3D Navier-Stokes equations, then we discuss the Boussinesq hypothesis
for such stochastic fluid equations, and finally we specialize the system in the 2D-3C setting to get the stochastic model
studied in the paper.

\subsection{Origin of noise in stochastic 3D Navier-Stokes equations}\label{sect heuristics}

Consider a 3D Newtonian viscous fluid described, in vorticity form, by the
equations%
\begin{align*}
\partial_{t}\omega+ v\cdot\nabla\omega-\omega\cdot\nabla v &  =\nu\Delta
\omega, \\
\omega &  =\operatorname{curl}v , \\
\omega|_{t=t_{0}} &  =\omega (t_{0} ) ,
\end{align*}
where $\nu>0$ is the fluid viscosity, $\omega$ is the vorticity field and $v$ the velocity field. In the
rigorous analysis below we shall assume that the space variable is in a torus,
in order to avoid the difficult question of vorticity at a solid boundary and
inessential troubles with the lack of compactness of full space.

We have taken as initial condition a time $t_{0}$ because the heuristic idea
described here is that, if at a certain time $t_{0}$ the vorticity
$\omega\left(  t_{0}\right)  $ is the sum of a large-scale component
$\omega_{L}\left(  t_{0}\right)  $ plus a small-scale component $\omega
_{S}\left(  t_{0}\right)  $ made of several small vortex structures $\omega_{S}%
^{i}\left(  t_{0}\right)  $:
\[
\omega\left(  t_{0}\right)  =\omega_{L}\left(  t_{0}\right)  +\sum_{i=1}%
^{N}\omega_{S}^{i}\left(  t_{0}\right),
\]
then, at least on a short time interval $\left[  t_{0},t_{0}+\Delta\right]  $,
the system
\begin{align*}
\partial_{t}\omega_{L}+ v\cdot\nabla\omega_{L}-\omega_{L}\cdot\nabla v &
=\nu\Delta\omega_{L}, \\
\partial_{t}\omega_{S}^{i}+ v\cdot\nabla\omega_{S}^{i}-\omega_{S}^{i}%
\cdot\nabla v &  =\nu\Delta\omega_{S}^{i}, \\
\omega_{L}|_{t=t_{0}}  =\omega_{L}( t_{0}) , \quad\omega
_{S}^{i}|_{t=t_{0}} &=\omega_{S}^{i}( t_{0})
\end{align*}
represents quite well the evolution of the different vortex structures. This
kind of approach has been used for instance to investigate the small
vortex-blob limit to point vortices \cite{MarPul} in 2D. The idea is that for a short while
the small vortex structures $\omega_{S}^{i}$ maintain their integrity, as
individual objects, before complicated merging and instability mechanisms take
place. The system above is equivalent to the original one, in the sense that if $\big(\omega_L, \omega_S^1, \ldots, \omega_S^N \big)$ is a solution of the system, then $\omega= \omega_L + \sum_{i=1}^N \omega_S^i$ is a solution of the original equation.

The next step is considering only the equation for the large scales, isolating
the term which is not closed, namely depends on the small scales:%
\[
\partial_{t}\omega_{L}+ v_{L}\cdot\nabla\omega_{L}-\omega_{L}\cdot\nabla
v_{L}-\nu\Delta\omega_{L}=-\sum_{i=1}^{N}\left(  v_{S}^{i}\cdot\nabla
\omega_{L}-\omega_{L}\cdot\nabla v_{S}^{i}\right)  .
\]
Here $v_{L}$, with $\operatorname{div} v_{L}=0$, has the property
$\operatorname{curl} v_{L}=\omega_{L}$; namely $v_{L}$ is reconstructed from
$\omega_{L}$ by the Biot-Savart law: $v_L =K\ast \omega_L$ where $K$ is the Biot-Savart kernel; 
and $v_{S}^{i}$ corresponds to $\omega_{S}^{i}$ in the same way.

Until now we have not modified the true Navier-Stokes equations, just
rewritten in a suitable way. The reformulation may not have the desired
properties when the vortex structures $\omega_{S}^{i}$ start to merge or the
large-scale field $\omega_{L}$ starts to develop small-scale instabilities,
but it is reasonable to expect that in a short time the structure is roughly
maintained.

Now we make a relevant approximation, whose validity is a major open problem.
We idealize the small-scale structures $v_{S}^{i}$ by a stochastic process, a
priori given, delta correlated in time. The idealization of delta correlation
in time is very strong, vaguely motivated by the very short time-scale of
$\omega_{S}^{i}$ compared to $\omega_{L}$. The idealization of Gaussianity and
independence on $\omega_{L}$ are strong limitations which should be better
investigated in the future. A large body of rigorous activity exists, see for
instance \cite{DebPapp, FP}, aimed to justify this approximation starting from the original
model, suitably perturbed or simplified.

Essentially, the modeling assumption made below replaces the small-scale
velocity field $\sum_{i=1}^{N} v_{S}^{i}\left(  t,x\right)  $ by
\[
 \sum_{k}\sigma_{k} ( x )  \frac{dW_{t}^{k}}{dt},
\]
where $\{\sigma_{k} \}_k$ are suitable divergence free fields and $\{W_{t}^{k} \}_k$ are
independent scalar Brownian motions. In the replacement, Stratonovich
integrals are used, in accordance to the Wong-Zakai principle (see rigorous
results in \cite{DebPapp, FP}).

\subsection{Boussinesq hypothesis and its difficulty in 3D fluids} \label{subsect def Boussinesq}

The presentation of this subsection is heuristic, the purpose being only to
describe what we mean by Boussinesq hypothesis for the class of stochastic 3D Navier-Stokes equations derived above,
and to explain which is the difficulty to solve the problem
in general (motivating the detailed analysis of the particular case treated in
this work). From now on we omit
the subscript $L$ for the large scales, since many other indices will appear.
But the reader should remember, as a matter of interpretation, that $v$ and $\omega$
are large-scale fields, $\{ \sigma_{k} \}_k$ describe the small-scale turbulent
components. We also conventionally set $t_{0}=0$.  
To simplify notation, for two vector fields $X$ and $Y$, we write $\mathcal{L}_{X} Y$ for the
Lie derivative $\mathcal{L}_{X}Y= X\cdot\nabla Y-Y\cdot\nabla X$.

Motivated by the heuristic discussions in the previous subsection, we consider the following stochastic 3D Navier-Stokes equations
\begin{align*}
d\omega_{\epsilon}+\left(  \mathcal{L}_{v_{\epsilon}}\omega_{\epsilon}%
-\nu\Delta\omega_{\epsilon}\right)  dt  & =-\sum_{k}\mathcal{L}_{\sigma
_{k}^{\epsilon}}\omega_{\epsilon}\circ dW_{t}^{k}, \\
\omega_{\epsilon}|_{t=0}  & =\omega^{0}%
\end{align*}
where $\{ \sigma_{k}^{\epsilon} \}_k$ are given vector fields, now
parametrized by $\epsilon>0$, hence also the solution $\omega_{\epsilon}$
depends on $\epsilon$. The stochastic multiplication is understood in the Stratonovich sense, which
is formally equivalent to the following It\^o equation with a correction term:
\[
d\omega_{\epsilon}+\left(  \mathcal{L}_{v_{\epsilon}}\omega_{\epsilon}%
-\nu\Delta\omega_{\epsilon}\right)  dt =-\sum_{k}\mathcal{L}_{\sigma_{k}^{\epsilon}}\omega_{\epsilon}\, dW_{t}^{k} + \frac{1}{2}\sum_{k}\mathcal{L}%
_{\sigma_{k}^{\epsilon}}\mathcal{L}_{\sigma_{k}^{\epsilon}}\omega_{\epsilon} \, dt .
\]

\textit{We say that Boussinesq hypothesis holds} for a family of coefficients
$\{ \sigma_{k}^{\epsilon} \}_k  $ if $\omega_{\epsilon}$ converges as
$\epsilon\rightarrow0$, in an appropriate sense, to a solution $\overline{\omega}$
of the deterministic equation
\begin{align*}
\partial_{t}\overline{\omega}+\mathcal{L}_{\overline{v}}\overline{\omega}  &
=\left(  \nu\Delta+\mathcal{D}\right)  \overline{\omega}, \\
\overline{\omega}|_{t=0}  & =\omega^{0}%
\end{align*}
where $\overline{v}=K\ast\overline{\omega}$, $\mathcal{D}$ is a suitable
second order differential operator, limit of $\frac{1}{2}\sum_{k}%
\mathcal{L}_{\sigma_{k}^{\epsilon}}\mathcal{L}_{\sigma_{k}^{\epsilon}}$ as
$\epsilon\rightarrow0$. This operator represents the enhancement of viscosity
produced by the noise in the scaling limit, and possibly incorporating other
effects than viscosity, like the AKA effect. Loosely speaking, we may consider
the Boussinesq hypothesis as a mean field result.

Finding families $\{ \sigma_{k}^{\epsilon} \}_k $ such that the
iterated Lie derivative operator $\frac{1}{2}\sum_{k}\mathcal{L}_{\sigma
_{k}^{\epsilon}}\mathcal{L}_{\sigma_{k}^{\epsilon}}$ has a nontrivial and
interesting limit $\mathcal{D}$ is not difficult. The difficulty is proving
that the It\^{o} terms $\int_{0}^{T}\mathcal{L}_{\sigma_{k}^{\epsilon}}%
\omega_{\epsilon} (t)\,  dW_{t}^{k}$ converge to zero in a certain
sense (in the heuristic mean field vision, they are the fluctuations). And, to
be more precise, the difficulty in proving this result for the It\^{o} terms
lies in \textit{the control of} $\omega_{\epsilon}$. Imposing assumptions on the
noise covariance in order to have convergece to zero of expressions like
\[
\sum_{k}\int_{0}^{T}\left\langle f\left(  t\right)  ,\mathcal{L}_{\sigma
_{k}^{\epsilon}}\phi\right\rangle dW_{t}^{k}%
\]
for a given process $f$ and smooth test function $\phi$, assumptions compatible
with those leading to a non-trivial limit $\mathcal{D}$ of $\frac{1}{2}%
\sum_{k}\mathcal{L}_{\sigma_{k}^{\epsilon}}\mathcal{L}_{\sigma_{k}^{\epsilon}%
}$, is possible. But a control on $\omega_{\epsilon}$ is needed and this is
highly non-trivial for 3D models.

The difficulty is only marginally due to the nonlinear term $\mathcal{L}%
_{v_{\epsilon}}\omega_{\epsilon}$. The same difficulty arises in the
investigation of the Boussinesq hypothesis for a passive magnetic field
$M_{\epsilon}$:
\begin{align*}
dM_{\epsilon}-\eta\Delta M_{\epsilon}\, dt  & =-\sum_{k}\mathcal{L}_{\sigma
_{k}^{\epsilon}}M_{\epsilon}\circ dW_{t}^{k}, \\
M_{\epsilon}|_{t=0}  & =M^{0}%
\end{align*}
($\eta>0$). The mean field equation%
\begin{align*}
\partial_{t}\overline{M}  & =\left(  \eta \Delta+\mathcal{D}\right) \overline{M} ,\\
\overline{M}|_{t=0}  & =M^{0}%
\end{align*}
is, in this linear case, rigorously obtained as the limit of the average of
$M_{\epsilon}$ (see for instance \cite{KrauseRadler}). However, the difficulty
in the control of $M_{\epsilon}$ remains and it is due to the fact that the
random stretching terms $M_{\epsilon}\cdot\nabla\sigma_{k}^{\epsilon}$ have a
stronger and stronger effect as $\epsilon\rightarrow0$.

\subsection{The stochastic 2D-3C model and main result} \label{2D 3C reduction}

As in the last subsection, we start from the
following stochastic version of the 3D Navier-Stokes equations
\begin{equation}\label{sect-1.2-NSE}
\aligned
d\omega+\left( v\cdot\nabla\omega- \omega\cdot\nabla v -\nu\Delta\omega\right)
dt  & =-\sum_{k}\left(  \sigma_{k}\cdot\nabla\omega-\omega\cdot\nabla
\sigma_{k}\right)  \circ dW_{t}^{k}, \\
\omega|_{t=0}  & =\omega_{0}.
\endaligned
\end{equation}
We still consider this equation at a formal level (but only to shorten the
next preliminary steps, since this could be made rigorous).

Assume that, with the notation $x=(x_{1},x_{2},x_{3})$, the initial condition
and the noise are independent of $x_{3}$:
\[
\omega_{0}=\omega_{0}\left(  x_{1},x_{2}\right)  ,\qquad\sigma_{k}=\sigma
_{k}\left(  x_{1},x_{2}\right)  .
\]
However, all vector fields still take values in $\mathbb{R}^{3}$ and we denote
the canonical basis of $\mathbb{R}^{3}$ by $\{e_{1},e_{2},e_{3} \}$. Under the
previous assumption, a solution of the stochastic equations above is given by
\begin{align*}
v\left(  t,x_{1},x_{2}\right)    & =v_{H}\left(  t,x_{1},x_{2}\right)
+v_{3}\left(  t,x_{1},x_{2}\right)  e_{3}, \\
\omega\left(  t,x_{1},x_{2}\right)    & =\omega_{H}\left(  t,x_{1}%
,x_{2}\right)  +\omega_{3}\left(  t,x_{1},x_{2}\right)  e_{3},
\end{align*}
where $v_{H},\, \omega_{H}$ live in the horizontal plane, namely take values in
the span of $\{e_{1},e_{2}\}$. Notice that the relation $\omega= {\rm curl}\, v$ implies
\begin{align*}
\omega_{3}  & =\nabla^{\perp}_H \cdot v_{H}, \quad
\omega_{H} =\nabla^{\perp}_H v_{3} ,
\end{align*}
so the horizontal and vertical components are interlaced. Here, $\nabla^{\perp}_H= (\partial_2, -\partial_1)$ with $\partial_i= \frac{\partial}{\partial x_i}$. The model
incorporates some degree of helicity (vorticity and velocity are not
orthogonal). Decompose also the noise:
\[
W=\sum_{k\in \mathcal I}\sigma_{k} W^{k}=\sum_{k\in \mathcal I}\sigma_{k}^{H} W^{k}+\sum_{k\in
\mathcal I}\sigma_{k}^{3} W^{k}e_{3},
\]
where $\mathcal I$ is some index set. Then we get the model ($\nabla_H= (\partial_1, \partial_2)$ and $\Delta_H= \partial_1^2+ \partial_2^2$)
\begin{align*}
d\omega_{3}+\left( v_{H}\cdot\nabla_H \omega_{3}- \omega_{H}\cdot\nabla_H
v_{3}-\nu\Delta_H \omega_{3}\right)  dt  & =-\sum_{k}\left(  \sigma_{k}^{H}%
\cdot\nabla_H \omega_{3}-\omega_{H}\cdot\nabla_H \sigma_{k}^{3}\right)  \circ
dW_{t}^{k}, \\
dv_{3}+\left( v_{H}\cdot\nabla_H v_{3}-\nu\Delta_H v_{3}\right)  dt  &
=-\sum_{k}\sigma_{k}^{H} \cdot\nabla_H v_{3}\circ dW_{t}^{k}, \\
\omega_{3}|_{t=0}  & =\omega_{3}^{0},\quad v_{3}|_{t=0}=v_{3}^{0} ,
\end{align*}
with obvious meaning of $\omega_{3}^{0}$ and $v_{3}^{0}$, given $\omega_{0}$.
Indeed, taking curl of the second equation, and using the simple identity
  \begin{equation*}
  \nabla_H^\perp( u\cdot\nabla_H v_{3} )= u \cdot\nabla_H \omega_H - \omega_H \cdot\nabla_H u
  \end{equation*}
which holds for any 2D divergence free vector field $u$, we obtain the first two components of the initial vorticity equations \eqref{sect-1.2-NSE}. We remark that the stretching term $\omega_{H}\cdot\nabla_H v_{3}$ in the equation for $\omega_3$ vanishes because $\omega_{H}= \nabla^\perp_H v_3$ is orthogonal to $\nabla_H v_{3}$.

In the next sections we make a rigorous analysis of this stochastic 2D-3C model and its scaling limit. Roughly speaking, under a suitable scaling limit of the noise introduced in the next section, we will show that the above stochastic model converges weakly to the following deterministic system
  \begin{equation}\label{2D-3C-limit}
  \left\{ \aligned
  \partial_t\omega_{3} + v_{H} \cdot \nabla_H \omega_{3} &= \big(\nu \Delta_H +\mathcal L_{\bar Q} \big)\omega_3 + \nabla_H \cdot (A\, \omega_{H}),\\
  \partial_t v_3+ v_{H} \cdot \nabla_H v_3 &= \big(\nu \Delta_H +\mathcal L_{\bar Q} \big) v_3,
  \endaligned \right.
  \end{equation}
where $\bar Q$ and $A$ are $2\times 2$ constant matrices, $\bar Q$ being nonnegative definite and  $\mathcal L_{\bar Q} f = \frac12 \nabla_H \cdot \big[\bar Q\, \nabla_H f \big]$, see Theorem \ref{thm:Boussinesq-general-result} below for more details. The system \eqref{2D-3C-limit} can be rewritten in the more familiar vorticity formulation as follows: taking curl of the second equation leads to
  $$\partial_t \omega_H+ v_H \cdot \nabla_H \omega_H - \omega_H \cdot \nabla_H v_H = \big(\nu \Delta_H +\mathcal L_{\bar Q} \big)\, \omega_H; $$
recalling that $\omega_H \cdot\nabla_H v_3 =0$, we combine the first equation in \eqref{2D-3C-limit} and the above one to obtain
  $$\partial_t \omega + v_H\cdot\nabla_H \omega- \omega_H \cdot\nabla_H v = \big(\nu \Delta_H +\mathcal L_{\bar Q} \big)\, \omega + \nabla_H \cdot (A\, \omega_{H})\, e_3 .$$
Since $\partial_3 \omega=\partial_3 v=0$, the nonlinear part can be further written as $v\cdot\nabla \omega- \omega \cdot\nabla v$ with $\nabla=(\partial_1, \partial_2, \partial_3)$. We see that there is an extra dissipation term $\mathcal L_{\bar Q}\, \omega$ predicted by Boussinesq's hypothesis, and a first order term $\nabla_H \cdot (A\, \omega_{H}) e_3$ responsible for the possible AKA effect; the latter is a scalar object multiplied by $e_3$, but we may also write it in the general 3D form, which reduces to the above one due to special cancellations of our set up. \smallskip

We conclude the introduction with the organization of the paper. We introduce in Section \ref{sect-vortex-noise} the
vortex noise which will be used to perturb the model. Section 3 contains
the statements of our main results, including the well-posedness of stochastic
2D-3C models and a scaling limit result to the deterministic model mentioned above. We prove these results in Sections 4 and 5, respectively.

\section{The vortex noise} \label{sect-vortex-noise}

In this section let us describe in detail the vortex noise $W$ that we will use to perturb the 2D-3C model introduced above.

\subsection{The covariance function} \label{subs-covariance-funct}

Let $\T^2= \R^2/\Z^2$ be the two dimensional torus; we may identify it as $[-\frac12, \frac12]^2$ with periodic boundary condition. Recall that for $i\in \{1,2\}$, $\partial_i= \frac\partial{\partial x_i}$, $\nabla_H=(\partial_1, \partial_2)$ and $\nabla_H^\perp = (\partial_2, -\partial_1)$.

Let us assume that the velocity field of a vortex takes the form
  $$\sigma( x) =\sigma_{H} (x) +\sigma_{3} (x)\, e_{3},\quad x\in \T^2, $$
where $\sigma_{H} (x)= (\sigma_1 (x), \sigma_2 (x))$ is a divergence free vector field on $\T^2 $; $\{\sigma_i\}_{i=1}^3$ are smooth periodic functions. For simplicity, we shall assume the vortex to be symmetric, namely,
  \begin{equation}\label{vortex-symmetry}
  \sigma_{H} (-x) = - \sigma_{H} (x), \quad \sigma_{3} (-x) = \sigma_{3} (x), \quad x\in \T^2.
  \end{equation}
As a typical example, we may take
  \begin{equation}\label{typical-example}
  \sigma_{H}= \Gamma K\ast \theta, \quad \sigma_3 = \gamma G\ast \chi,
  \end{equation}
where $\Gamma, \gamma>0$ are vortex intensities, $K$ is the Biot-Savart kernel and $G$ the Green function on $\T^2$, and $\theta, \chi\in C_c^\infty\big((-\frac12, \frac12)^2 \big)$ are symmetric functions.

The associated vorticity field
  $$\omega(x) =\omega_{H} (x) +\omega_{3} (x)\, e_{3} $$
is determined by the following relations:
  $$\omega_{H} = \nabla_H^\perp \sigma_3, \quad \omega_3= \nabla_H^\perp\cdot \sigma_{H}. $$

Let $X$ be a uniform random variable with values in $\T^{2}$, we can now define the velocity field of the random 3D vortex as follows:
  $$\Sigma(x)= \sigma(x - X) = \sigma_{H}(x - X) + \sigma_{3} (x - X)\, e_{3}, \quad x\in \T^2 .$$

\begin{remark}
We describe here the ``intuitive picture'' of the vortex noise: choose a random location $X$ in the
horizontal torus $\mathbb{T}^{2}$, create a vertical 3D vortex object,
similar to a sort of ``vertical cylinder'', with a certain velocity component in the vertical direction.
Repeating this choice at high frequency, choosing at random the location $X$ and the orientation (namely, multiply $\omega(x)$ by $\pm 1$), we get a ``random walk'' of vortex structures that converges to a noise in a suitable scaling limit, see \cite{FH}.
\end{remark}

Consider the covariance matrix function of the random field $\Sigma$:%
\begin{align*}
Q ( x,y) & =\mathbb{E}[\Sigma (x)  \otimes \Sigma (y) ]
=\mathbb{E} [ \sigma ( x-X ) \otimes \sigma( y-X ) ] =\int_{\T^{2}} \sigma ( x-z ) \otimes \sigma (y-z)\, dz.
\end{align*}
It is space-homogeneous, that is, equal to $Q (x-y)$ with
\begin{equation}\label{eq:covariance-funct}
Q (a)  =\int_{\T^{2}} \sigma(a-w) \otimes\sigma(-w) \, dw,
\end{equation}
where we use the change of variables $y-z=-w$ in the integral above. As a matrix valued function, $Q$ is smooth on $\T^2$.

We decompose $Q(a)$ in the form%
\[
Q(a)  = \begin{pmatrix}
Q_{H}(a) & Q_{H,3}(a) \\
Q_{3,H}(a) & Q_{3}(a)
\end{pmatrix},
\]
where, using \eqref{vortex-symmetry},
\begin{equation}\label{covariance-functions}
\aligned
Q_{H}(a) &= -\int_{\mathbb{T}^{2}} \sigma_{H} (a-x) \otimes \sigma_{H} (x)\, dx, \\
Q_{3}(a) &=\int_{\mathbb{T}^{2}} \sigma_{3}(a-x) \otimes \sigma_{3}(x) \, d x,\\
Q_{H,3}(a) &=\int_{\mathbb{T}^{2}} \sigma_{H} (a-x) \otimes\sigma_{3} (x) \, dx ,\\
Q_{3,H} (a) &= -\int_{\mathbb{T}^{2}} \sigma_{3}(a-x) \otimes \sigma_{H} (x) \, dx.
\endaligned
\end{equation}

\begin{remark}
Recall that, by the definition of covariance,%
\[
Q (-a)  =Q ( a)^{\ast}%
\]
and that mirror symmetry means
\[
Q ( -a)  =Q (a),
\]
which also means that $Q (a)$ is a symmetric matrix.
Notice that the covariance matrix of a random vector is always positive definite but, in spite
of the name, $Q (a)  $ is not the covariance matrix
of a vector but the mutual covariance between two vectors, which may have any sign.
\end{remark}

We have, by \eqref{vortex-symmetry},
\begin{align*}
Q_{H} (-a) & = -\int_{\mathbb{T}^{2}} \sigma_{H} (-a-x) \otimes \sigma_{H} (x)\, dx =\int_{\mathbb{T}^{2}} \sigma_{H} (a+x) \otimes \sigma_{H} (x)\,  dx;
\end{align*}
changing variable $z= -x$ yields
\begin{align*}
Q_{H} (-a) & =\int_{\mathbb{T}^{2}} \sigma_{H} (a-z) \otimes \sigma_{H}( -z)\, dz =Q_{H} (a).
\end{align*}
Similarly,%
\begin{align*}
Q_{3} (-a) & =Q_{3} (a), \\
Q_{3,H}(-a) & = -Q_{3,H} (a), \\
Q_{H,3} (-a) &= -Q_{H,3} (a) ,
\end{align*}
hence \textit{mirror symmetry} does not hold.

\begin{remark}[Diagonality of $Q_H(0)$] \label{rem-diagonal-horizontal-covar}
Let the horizontal velocity field $\sigma_H$ be given as in \eqref{typical-example}; by the definition of horizontal covariance function $Q_H$, we have
  $$ \aligned
  Q_H(0) &= \int_{\T^2} \sigma_H(x) \otimes \sigma_H(x)\, dx= \Gamma^2 \int_{\T^2} (K\ast\theta)(x) \otimes (K\ast\theta)(x)\, dx \\
  &= \Gamma^2 \int_{\T^2} (\nabla_H^\perp G\ast\theta)(x) \otimes (\nabla_H^\perp G\ast\theta)(x)\, dx.
  \endaligned $$

If we assume that $\theta$ is \emph{radially symmetric}, then $Q_H(0) = 2\kappa I_2$ for some $\kappa>0$, here $I_2$ is the $2\times 2$ unit matrix. Indeed, we have
  $$Q_H^{1,1}(0)= \Gamma^2 \int_{\T^2} \big[(\partial_2 G\ast\theta)(x) \big]^2 \, dx, \quad Q_H^{2,2}(0)= \Gamma^2 \int_{\T^2} \big[(\partial_1 G\ast\theta)(x) \big]^2 \, dx. $$
Using the radial symmetry of $\theta$, one can easily show that
  $$(\partial_2 G\ast\theta)(x_2, x_1) = (\partial_1 G\ast\theta)(x_1, x_2) ,$$
therefore, changing variable $(x_1, x_2)\to (x_2, x_1)$, we have
  $$Q_H^{1,1}(0)= \Gamma^2 \int_{\T^2} \big[(\partial_2 G\ast\theta)(x_2, x_1) \big]^2 \, dx_2 dx_1 = \Gamma^2 \int_{\T^2} \big[(\partial_1 G\ast\theta)(x_1, x_2) \big]^2 \, dx_1 dx_2= Q_H^{2,2}(0).  $$
Next, for the off-diagonal entries,
  $$Q_H^{1,2}(0) = - \Gamma^2 \int_{\T^2} (\partial_2 G\ast\theta)(x) \,(\partial_1 G\ast\theta)(x) \, dx = Q_H^{2,1}(0). $$
Again by the radial symmetry of $\theta$, one can show that
  $$\aligned
  (\partial_1 G \ast \theta) (-x_1, x_2)= -(\partial_1 G \ast \theta) (x_1, x_2),  &\quad (\partial_1 G \ast \theta) (x_1, -x_2) = (\partial_1 G \ast \theta) (x_1, x_2), \\
  (\partial_2 G \ast \theta) (-x_1, x_2)= (\partial_2 G\ast \theta) (x_1, x_2),  &\quad (\partial_2 G\ast \theta) (x_1, -x_2) = -(\partial_2 G \ast \theta) (x_1, x_2).
  \endaligned $$
These properties immediately imply that $Q_H^{1,2}(0)= Q_H^{2,1}(0) =0$. Thus the assertion holds with $\kappa = \frac12 Q_H^{i,i}(0)= \frac14 {\rm Tr}(Q_H(0))$.
\end{remark}

\begin{remark}\label{rem-estimates-horizontal-covar}
Later on we will make a scaling of the vortex noise: for $\ell\in (0,1)$, let $\theta_\ell(x)= \ell^{-2} \theta(\ell^{-1} x)$ and define $\sigma_H^\ell = \Gamma K\ast \theta_\ell$. Then, as $\ell\to 0$, $\sigma_H$ approaches a point vortex. In this case, if $\theta$ is a probability density with compact support in $(-\frac12, \frac12)^2$, it is a classical result that its energy is of order $\log \ell^{-1}$; more precisely,
  \begin{equation}\label{eq:trace-lower-bound}
  {\rm Tr}(Q_H(0)) \sim \frac{\Gamma^2}{4\pi} \log \ell^{-1} \quad \mbox{as } \ell\to 0,
  \end{equation}
see Proposition \ref{prop:limit-matrix} and its proof in Section \ref{subsec-proof-prop}, for instance, Lemma \ref{lem-limit-integral-G}.
\end{remark}

Let $\mathbb Q_H$ (resp. $\mathbb Q_3$) be the covariance operator associated to the covariance matrix $Q_H$ (resp. covariance function $Q_3$) defined above. We conclude this part with estimates on the operator norms of $\mathbb Q_H $ and $ \mathbb Q_3$ when the vortex velocity field is given by \eqref{typical-example}. Recall that the covariance operators are defined as follows:
  $$\aligned
  (\mathbb Q_H V)(x) &= \int_{\T^2} Q_H(x-y)V(y)\, dy = (Q_H\ast V)(x) , \quad x\in \T^2,\, V\in L^2(\T^2,\R^2); \\
  (\mathbb Q_3 f)(x) &= \int_{\T^2} Q_3(x-y)f(y)\, dy = (Q_3\ast f)(x) , \quad x\in \T^2,\, f\in L^2(\T^2,\R).
  \endaligned $$

\begin{lemma}\label{lem-covar-operator}
Assume that the vortex velocity field is defined as in \eqref{typical-example}, then one has
  $$\|\mathbb Q_H \|_{L^2\to L^2} \le \Gamma^2 \| K \|_{L^1}^2 \| \theta \|_{L^1}^2 \quad \mbox{and} \quad \|\mathbb Q_3 \|_{L^2\to L^2} \le \gamma^2 \| G \|_{L^1}^2 \| \chi \|_{L^1}^2. $$
\end{lemma}

\begin{proof}
We only prove the first estimate since the other proof is similar. By the classical inequality for convolution, one has
  $$\|\mathbb Q_H V \|_{L^2} = \|Q_H\ast V \|_{L^2} \le \|Q_H \|_{L^1} \| V \|_{L^2} $$
which implies $\|\mathbb Q_H \|_{L^2\to L^2} \le \|Q_H \|_{L^1}$.

Recalling the definition of $Q_H$, we have
  $$\aligned
  \|Q_H \|_{L^1} &= \int_{\T^2} \bigg|\int_{\mathbb{T}^{2}} \sigma_H(x-y)  \otimes \sigma_H
  (y)\, dy \bigg|\, dx \\
  &\le \int_{\T^2} |\sigma_H (y)|\, dy \int_{\mathbb{T}^{2}} | \sigma_H(x-y) |\, dx \\
  &= \|\sigma_H \|_{L^1}^2.
  \endaligned $$
Note that $\sigma_H= \Gamma K\ast \theta$, we use again the convolutional inequality and get
  $$\|Q_H \|_{L^1} \le \Gamma^2 \| K\ast \theta\|_{L^1}^2 \le \Gamma^2 \| K \|_{L^1}^2\| \theta \|_{L^1}^2. $$
Combining the above estimates we finish the proof.
\end{proof}

\subsection{Series expansion of noise} \label{subsec-series-noise}

The smoothness of $Q$ implies that the corresponding covariance operator $\mathbb{Q}$ is of trace class. Let $\{f_k\}_k$ be a CONS of $\mathcal H=L^2_{\rm sol} (\T^3, \R^3)$ (sol means solenoidal fields) consisting of eigenvectors of $\mathbb{Q}$:
  $$\mathbb{Q} f_k = \lambda_k f_k, \quad \lambda_k \ge 0, \quad k\ge 1. $$
For simplicity, we assume $\lambda_k >0$ for all $k\ge 1$. From the above identity we deduce that the field $f_k$ depends only on the horizontal variable $x=(x_1, x_2)$. Indeed, for any $\tilde x=(x, x_3)\in \T^3$, by the definition of $\mathbb{Q}$,
  $$\aligned
  \lambda_k f_k(\tilde x)&= \int_{\T^3} Q(\tilde x-\tilde y) f_k(\tilde y) \, d \tilde y= \int_{\T^3} Q(x -y) f_k(\tilde y) \, d \tilde y \\
  &= \int_{\T^2} Q(x -y)\bigg[ \int_{\T} f_k(y,y_3) \, d y_3 \bigg] dy,
  \endaligned $$
where $y$ is the horizontal part of $\tilde y\in \T^3$. The right-hand side depends only on $x$, thus $f_k(\tilde x) = f_k(x, x_3)$ is independent of $x_3$; this implies $\int_{\T} f_k(y,y_3) \, d y_3 = f_k(y)$. The above equality can be rewritten as
  $$\lambda_k f_k(x)= \int_{\T^2} Q(x -y) f_k(y) \, d y. $$
The property of convolution implies that $f_k$ is also smooth on $\T^2$.

Now let $W=\{W(t) \}_{t\ge 0}$ be a $\mathbb{Q}$-Brownian motion on $\mathcal H$; then there exists a sequence of independent real Brownian motions $\{W^k\}_{k\ge 1}$ such that
  $$W(t)= \sum_k \sqrt{\lambda_k}\, W^k_t f_k = \sum_k \sigma_k W^k_t, $$
where we have denoted for simplicity that $\sigma_k= \sqrt{\lambda_k} f_k,\, k\ge 1$. We remark that
  \begin{equation}\label{covariance-matrix-expansion}
  Q(x -y) = \sum_k \sigma_k(x)\otimes \sigma_k(y)\quad \mbox{for all } x, y\in \T^2,
  \end{equation}
and, since $Q$ is smooth, the series converges in $C^m(\T^2\times \T^2)$ for any $m\ge 1$.
Similarly, we write the field $\sigma_k$ as
  $$\sigma_k(x) = \sigma_k^H(x) + \sigma_k^3(x)\, e_3 ,$$
where the horizontal part $\sigma_k^H = (\sigma_k^1, \sigma_k^2)$ is a divergence free field on $\T^2$. Accordingly, we can write the noise $W(t)$ as
  $$\aligned
  W(t,x)&= W_H(t,x) + W_3(t,x)\, e_3 , \\
  W_H(t,x)&= \sum_k \sigma_k^H(x) W^k_t, \quad W_3(t,x)= \sum_k \sigma_k^3(x) W^k_t.
  \endaligned $$

\section{Main results}

In this part we first show in Section \ref{subsec-well-posedness} the well-posedness of the stochastic 2D-3C models perturbed by the vortex noise introduced above, then in Section \ref{subsec-scaling} we shall rescale the noise and show that the stochastic models converge to deterministic 2D-3C model which might exhibit AKA (short for anisotropic kinetic alpha) effect.

In the sequel, since the space variables $x$ are always two dimensional, we omit for simplicity the subscript $H$ in $\Delta_H,\, \nabla_H,\, \nabla_H^\perp$ etc, but we still write $v_H,\, \omega_H$ to distinguish them from $v_3,\, \omega_3$. We denote the usual Sobolev spaces on $\T^2$ as $H^s= H^s(\T^2)$, which will also be used for vector valued functions. When we want to specify the target space, we write more explicitly as $H^s(\T^2,\R^d)$.  The norm in $H^s$ is written as $\|\cdot \|_{H^s}$ and in case $s=0$, we write $H^0$ as $L^2$ with the norm $\|\cdot \|_{L^2}$. We assume the spaces $H^s$ consist of functions with zero spatial mean on $\T^2$. The notation $\<\cdot, \cdot\>$ will be used for the scalar product in $L^2$ or the duality between $H^s$ and $H^{-s}$. Sometimes we denote the norm in $L^p(0,T; H^s)$ simply as $\|\cdot \|_{L^p H^s}$, $p\ge 1, s\in \R$. In the sequel, we write $a\lesssim b$ if there exists some unimportant constant $C>0$ such that $a\le Cb$; the notation $\lesssim_{\nu, T}$ means that the constant $C$ is dependent on $\nu,\, T$.

\subsection{Well-posedness of stochastic 2D-3C models} \label{subsec-well-posedness}

Let $(\Omega,\mathcal F, (\mathcal F_t), \P)$ be a filtered probability space and $\{W^k\}_k$ a family of independent standard Brownian motions. Recall the vector fields $\{\sigma_k^H\}_k$ and functions $\{\sigma_k^3\}_k$ introduced in Section \ref{subsec-series-noise}. Motivated by the discussions in Section \ref{2D 3C reduction}, we consider the following stochastic 2D-3C model:
  \begin{equation}\label{eq:stoch-2D-3C-model}
  \left\{ \aligned
  d\omega_{3} + (v_{H} \cdot \nabla\omega_{3} - \nu\Delta \omega_3)\, dt &= - \sum_k \big(\sigma_k^H \cdot \nabla \omega_{3} - \omega_{H}\cdot \nabla \sigma_k^3 \big)\circ dW^k_t, \\
  dv_3 + (v_{H} \cdot \nabla v_3 - \nu\Delta v_3)\, dt &= - \sum_k \sigma_k^H \cdot\nabla v_{3} \circ dW^k_t.
  \endaligned  \right.
  \end{equation}
where $\circ d$ means Stratonovich stochastic differential. We recall that
  $$\omega_3= \nabla^\perp\cdot v_H, \quad \omega_H= \nabla^\perp v_3; $$
the latter identity explains why the stretching term $\omega_H \cdot \nabla v_3$ vanishes in the first equation; however, the random stretching terms $\sum_k \omega_{H}\cdot \nabla \sigma_k^3 \circ dW^k_t$ remain and, depending on certain condition, they might be the origin of the AKA term. Moreover, by the Biot-Savart law, $v_{H}= K\ast \omega_3$ where $K$ is the Biot-Savart kernel on $\T^2$. In It\^o formulation, the system becomes (see Section \ref{subsec-Ito-corrector} for the derivations)
  \begin{equation}\label{eq:stoch-2D-3C-model-Ito}
  \left\{ \aligned
  d\omega_{3} + (v_{H} \cdot \nabla\omega_{3} - \nu\Delta \omega_3)\, dt
  & = - \sum_k \big(\sigma_k^H \cdot \nabla \omega_{3} -\omega_{H}\cdot \nabla \sigma_k^3 \big)\, dW^k_t\\
  &\quad + \big(\mathcal L \omega_3 + \nabla \cdot [\nabla Q_{H,3}(0)\, \omega_{H}] \big)\, dt, \\
  dv_3 + (v_{H} \cdot \nabla v_3 - \nu\Delta v_3)\, dt &= - \sum_k \sigma_k^H \cdot\nabla v_{3} \, dW^k_t + \mathcal L v_3\, dt,
  \endaligned  \right.
  \end{equation}
where $\nabla Q_{H,3}(0)$ is a constant matrix given by
  \begin{equation}\label{eq:nabla-Q-H-3}
  \nabla Q_{H,3}(0) = \begin{pmatrix}
  \partial_1 Q_{1,3}(0) &\, \partial_2 Q_{1,3}(0) \\
  \partial_1 Q_{2,3}(0) &\, \partial_2 Q_{2,3}(0)
  \end{pmatrix}
  \end{equation}
and $\mathcal L$ is a second order differential operator given by
  $$
  \mathcal L f = \frac12 \sum_k \sigma_k^H \cdot \nabla (\sigma_k^H \cdot \nabla f), \quad f\in C^2(\T^2) .
  $$
As the vector fields $\{\sigma_k^H \}_k$ are divergence free, we have
  \begin{equation}\label{eq:corrector-L}
  \mathcal L f = \frac12 \sum_k \nabla \cdot \big[(\sigma_k^H \cdot \nabla f) \sigma_k^H\big] = \frac12 \sum_k \nabla \cdot \big[(\sigma_k^H \otimes \sigma_k^H) \nabla f \big]= \frac12 \nabla \cdot \big[Q_H(0) \nabla f \big],
  \end{equation}
where the last step is due to horizontal part of \eqref{covariance-matrix-expansion}.

We first give the meaning of solutions to the system \eqref{eq:stoch-2D-3C-model-Ito}.

\begin{definition}\label{sect-5-definition}
Let $(\omega_3^0, v_3^0)\in (L^2)^2$. A pair of stochastic processes $(\omega_3, v_3)$ defined on some stochastic basis $(\Omega,\mathcal F, (\mathcal F_t), \P)$ with $(\mathcal F_t)$-Brownian motions $\{W^k\}_k$, which are $(\mathcal F_t)$-predictable and have trajectories in
  $$L^\infty(0,T; L^2) \cap L^2(0,T; H^1), $$
is called a weak solution to \eqref{eq:stoch-2D-3C-model-Ito} if for any $\phi\in H^1(\T^2)$, $\P$-a.s. for all $t\in [0,T]$, one has
  \begin{equation} \label{eq:defin-1}
  \aligned
  \<\omega_3(t),\phi\> &= \<\omega_3^0,\phi\> + \int_0^t \<\omega_3(s), v_{H}(s)\cdot\nabla \phi \>\, ds -\int_0^t \big\<\omega_{H}(s), (\nabla Q_{H,3}(0))^\ast \nabla\phi \big\> \, ds  \\
  &\quad - \int_0^t \Big\<\nabla\omega_3(s), \Big(\nu I_2+ \frac12 Q_H(0)\Big) \nabla \phi \Big\>\, ds \\
  &\quad + \sum_k \int_0^t \big[ \big\<\omega_3(s), \sigma_k^H \cdot \nabla\phi \big\> - \big\<\sigma_k^3, \omega_{H}(s)\cdot \nabla \phi\big\> \big]\, dW^k_s ,
  \endaligned
  \end{equation}
  \begin{equation} \label{eq:defin-2}
  \aligned
  \<v_3(t),\phi\> &= \<v_3^0,\phi\> + \int_0^t \<v_3(s), v_{H}(s)\cdot\nabla \phi\>\, ds + \sum_k \int_0^t \big\<v_3(s), \sigma_k^H \cdot \nabla\phi \big\>\, dW^k_s\\
  &\quad -\int_0^t \Big\<\nabla v_3(s), \Big(\nu I_2 + \frac12 Q_H(0)\Big) \nabla\phi \Big\>\, ds.
  \endaligned
  \end{equation}
\end{definition}

We remark that if $\omega_3, v_3\in L^2\big(\Omega, L^\infty(0,T; L^2) \cap L^2(0,T; H^1)\big)$, then all the terms in \eqref{eq:defin-1} and \eqref{eq:defin-2} are well defined. According to the above definition, any solution $(\omega_3, v_3)$ can be decomposed as
  $$\omega_3(t)= \omega_3^0 + V_\omega(t) + M_\omega(t) , \quad v_3(t)= v_3^0 + V_v(t) + M_v(t), $$
where $V_\omega, \, V_v$ take values in $W^{1,2}(0,T; H^{-1})$, while $M_\omega,\, M_v$ are $L^2$-valued continuous local martingales. For instance, one has
  $$\aligned
  V_\omega(t)&= \int_0^t \big( -v_{H} \cdot \nabla\omega_{3} + \nu\Delta \omega_3+ \mathcal L \omega_3 + \nabla \cdot [\nabla Q_{H,3}(0)\, \omega_{H}] \big)(s)\, ds,\\
  M_\omega(t)&= - \sum_k \int_0^t \big(\sigma_k^H \cdot \nabla \omega_{3} -\omega_{H}\cdot \nabla \sigma_k^3 \big)(s)\, dW^k_s;
  \endaligned $$
similar (and simpler) expressions hold for $V_v(t)$ and $M_v(t)$. Using the regularity properties of $\omega_3$ and $v_H= K\ast \omega_3$, it is easy to see that
  $$\aligned
  \| v_H(s) \cdot \nabla\omega_3(s) \|_{H^{-1}} &= \| \nabla \cdot (v_H(s) \omega_3(s)) \|_{H^{-1}} \lesssim \| v_H(s) \omega_3(s) \|_{L^2} \\
  &\lesssim \|\omega_3(s) \|_{L^2} \|v_H(s) \|_{H^2} \lesssim \|\omega_3 \|_{L^\infty L^2} \|\omega_3(s) \|_{H^1},
  \endaligned $$
and thus $\P$-a.s. $v_H \cdot \nabla\omega_3 \in L^2(0,T; H^{-1})$; the other terms in the expression of $V_\omega$ clearly enjoy the same property, therefore $\partial_t V_\omega \in L^2(0,T; H^{-1})$. The regularity on the martingale part $M_\omega$ can also be verified by classical estimates, cf. \cite[Lemma 3.2]{FGL21b} for detailed computations. From these discussions we see that the conditions in  \cite[Theorem 2.13]{RL18} are verified, thus the It\^o formula therein is applicable which implies that $t\mapsto \|\omega_3(t) \|_{L^2}$ is continuous. We conclude from this and the continuity of $\<\omega_3(t),\phi\>$ in $t$ that $\omega_3$ has continuous trajectories in $L^2$. The same result holds for $v_3$.

The purpose of this part is to show the well-posedness of the system \eqref{eq:stoch-2D-3C-model-Ito}. Let $\|\cdot \|_M$ be a norm on the space of $2\times 2$ matrices.

\begin{theorem}\label{thm:well-posedness}
Let $Q$ be the smooth covariance matrix function as defined in Section \ref{subs-covariance-funct}. For any $(\omega_3^0, v_3^0)\in (L^2)^2$, the stochastic 2D-3C model \eqref{eq:stoch-2D-3C-model-Ito} admits a probabilistically weak solution $(\omega_3, v_3)$ in the sense of Definition \ref{sect-5-definition}, satisfying the bounds
  \begin{align}
  & \P \mbox{-a.s. for all } t\in[0,T], \quad \|v_3(t)\|_{L^2}^2 + 2\nu \int_0^t \|\nabla v_3(s)\|_{L^2}^2\, ds \le \|v_3^0 \|_{L^2}^2, \label{thm:well-posedness.1} \\
  & \E\bigg[\sup_{t\in [0,T]} \|\omega_3(t)\|_{L^2}^2 + \nu \int_0^T \|\nabla \omega_3(s)\|_{L^2}^2\, ds \bigg] \le C_{\nu, Q} \big(\|v_3^0 \|_{L^2}^2 + \|\omega_3^0 \|_{L^2}^2\big), \label{thm:well-posedness.2}
  \end{align}
where $C_{\nu, Q}$ is a constant depending on $\nu, \|\nabla Q_{H,3}(0)\|_M, \|\nabla^2 Q_3(0)\|_M$.

Moreover, if $\|\nabla^2 Q_3(0)\|_M \le 2\nu$, then pathwise uniqueness holds for \eqref{eq:stoch-2D-3C-model-Ito} and thus the system has a unique probabilistically strong solution.
\end{theorem}

This result will be proved in Section \ref{sec-proof-well-posedness}. The existence part follows from the classical Galerkin approximations and compactness arguments, therefore, we only provide some a priori estimates needed for showing tightness of laws of approximating solutions. Concerning the proof of pathwise uniqueness, the discussions below Definition \ref{sect-5-definition} allow us to apply the It\^o formula in \cite[Theorem 2.13]{RL18}, then the uniqueness follows from some relatively standard arguments.

\subsection{Scaling limit for stochastic 2D-3C models} \label{subsec-scaling}

In this part, following the recent works \cite{FGL21a, Galeati20, Luo21a}, we want to take a suitable scaling limit of the noise $W$ in \eqref{eq:stoch-2D-3C-model}, and show that the solutions converge to some limit which solves a deterministic 2D-3C model with an extra dissipation term, and possibly also a first order term responsible for the AKA effect. The latter has been discussed in physics literature for a long time (see e.g. \cite{Wirth}), but escaped from rigorous mathematical treatment so far.

To this purpose, we shall take a family of vortex noise, defined as in Section \ref{subs-covariance-funct}, where the velocity field now depends on some parameter $\ell\in (0,1)$:
  $$\sigma^\ell(x)= \sigma_H^\ell(x) + \sigma_3^\ell(x)\, e_3 , \quad x\in \T^2,$$
satisfying the symmetry property \eqref{vortex-symmetry}; accordingly, we have the corresponding covariance matrix function
  $$Q^\ell(x) = \int_{\T^2} \sigma^\ell(x -y) \otimes \sigma^\ell(-y)\,d y = \begin{pmatrix}
  Q_H^\ell(x) &\, Q_{H,3}^\ell(x) \smallskip \\
  Q_{3,H}^\ell(x) &\, Q_3^\ell(x)
  \end{pmatrix}, \quad x\in \T^2, $$
where $Q_H^\ell, \, Q_{H,3}^\ell$ and $Q_3^\ell$ admit similar expressions as in \eqref{covariance-functions}. We can define the covariance operator $\mathbb Q^\ell$ associated to $Q^\ell$, acting on functions in $L^2(\T^2,\R^3)$ and with operator norm $\|\mathbb  Q^\ell\|_{L^2\to L^2}$; in the same way, we have the operators $\mathbb  Q_H^\ell$ and $\mathbb  Q_3^\ell$ with associated norms. Moreover, we denote
  $$\nabla Q_{H,3}^\ell(0) = \begin{pmatrix}
  \partial_1 Q_{1,3}^\ell(0) &\ \partial_2 Q_{1,3}^\ell(0) \smallskip \\
  \partial_1 Q_{2,3}^\ell(0) &\ \partial_2 Q_{2,3}^\ell(0)
  \end{pmatrix}. $$

Our basic assumptions on the covariance functions are as follows:
\begin{hypothesis} \label{basic-hypotheses}
\begin{itemize}
\item[(a)] the limit $\bar Q:= \lim_{\ell\to 0} Q_H^\ell(0)$ exists and is a nonnegative definite $2\times 2$ matrix;
\item[(b)] it holds that $\lim_{\ell\to 0} \|\mathbb  Q_H^\ell\|_{L^2\to L^2} \vee \|\mathbb  Q_3^\ell\|_{L^2\to L^2} =0$;
\item[(c)] the limit $A:= \lim_{\ell\to 0} \nabla Q_{H,3}^\ell(0)$ exists (might be a zero matrix);
\item[(d)] it holds that $\sup_{\ell\in (0,1)} \|\nabla^2 Q_3^\ell(0) \|_M <+\infty$.
\end{itemize}
\end{hypothesis}

Let $W^\ell= \{W^\ell(t)\}_{t\ge 0}$ be a $\mathbb Q^\ell$-Brownian motion in $\mathcal H$; similarly to Section \ref{subsec-series-noise}, we decompose $W^\ell$ as follows:
  $$\aligned
  W^\ell(t,x)&= W^\ell_H(t,x) + W^\ell_3(t,x)\, e_3 , \\
  W^\ell_H(t,x)&= \sum_k \sigma^{\ell,H}_k (x) W^k_t, \quad W^\ell_3(t,x)= \sum_k \sigma^{\ell,3}_k(x) W^k_t,
  \endaligned $$
where $\sigma^{\ell,H}_k= \big( \sigma^{\ell,1}_k, \sigma^{\ell,2}_k\big)$. Now we consider the stochastic 2D-3C model driven by $W^\ell$:
  $$\left\{ \aligned
  d\omega^\ell_{3} +\big(v^\ell_H \cdot \nabla \omega^\ell_{3}-\nu\Delta \omega^\ell_3\big)\, dt &= -\sum_k \big(\sigma^{\ell,H}_k \cdot \nabla \omega^\ell_{3} -\omega^\ell_H \cdot \nabla \sigma^{\ell,3}_k \big)\circ dW^k_t, \\
  dv^\ell_3+ \big(v^\ell_H \cdot \nabla v^\ell_3 -\nu\Delta v^\ell_3 \big)\, dt &= -\sum_k \sigma^{\ell,H}_k \cdot\nabla v^\ell_{3} \circ dW^k_t.
  \endaligned \right. $$
where $\omega^\ell_H= \nabla^{\perp} v^\ell_{3}$ and $\omega^\ell_{3}= \nabla^{\perp}\cdot v^\ell_H$. The associated It\^o formulation is
  \begin{equation}\label{eq:Ito-system-ell}
  \left\{ \aligned
  d\omega^\ell_{3} + v^\ell_H \cdot \nabla \omega^\ell_{3} \, dt &= -\sum_k \big(\sigma^{\ell,H}_k \cdot \nabla \omega^\ell_{3} - \omega^\ell_H\cdot \nabla \sigma^{\ell,3}_k \big)\, dW^k_t \\
  &\quad + \big(\nu \Delta +\mathcal L^\ell\big)\omega^\ell_3 \, dt + \nabla\cdot \big[\nabla Q_{H,3}^{\ell} (0)\, \omega^\ell_H \big] \, dt,\\
  dv^\ell_3+ v^\ell_H \cdot \nabla v^\ell_3 \, dt &= -\sum_k \sigma^{\ell,H}_k \cdot\nabla v^\ell_{3} \, dW^k_t + \big(\nu \Delta +\mathcal L^\ell \big) \omega^\ell_3 \, dt ,
  \endaligned \right.
  \end{equation}
where the operator $\mathcal L^\ell f = \frac12 \nabla\cdot \big[Q_H^\ell(0) \nabla f \big]$. Given an $\ell$-independent initial data $(\omega_3(0), v_3(0))$, under Hypothesis \ref{basic-hypotheses}, Theorem \ref{thm:well-posedness} asserts that the system has probabilistically weak solutions $\big(\omega^\ell_3, v^\ell_3\big)$, satisfying bounds as \eqref{thm:well-posedness.1} and \eqref{thm:well-posedness.2} which are uniform in $\ell\in (0,1)$. We remark that if we strengthen condition (d) in Hypothesis \ref{basic-hypotheses} to be $\sup_{\ell\in (0,1)} \|\nabla^2 Q_3^\ell(0) \|_M \le 2\nu$, then the solutions $\big(\omega^\ell_3, v^\ell_3\big)$ are also strong in the probabilistic sense.

Our main result reads as follows:

\begin{theorem}\label{thm:Boussinesq-general-result}
Assume Hypothesis \ref{basic-hypotheses}; for any $\ell\in (0,1)$, let $\big(\omega^\ell_3, v^\ell_3\big)$ be a probabilistic weak solution to \eqref{eq:Ito-system-ell} with the same initial data $(\omega_3(0), v_3(0))$. Then the family $\{\eta^\ell \}_{\ell\in (0,1)}$ of laws of $\big\{ \big(\omega^\ell_3, v^\ell_3\big) \big\}_{\ell\in (0,1)}$ is tight on $ L^2(0,T; L^2)$, and any weakly convergent subsequence of $\{\eta^\ell \}_{\ell\in (0,1)}$ converges to a limit probability measure which is supported on the weak solution $(\omega_3, v_3)$ of the deterministic 2D-3C model:
  \begin{equation}\label{eq:2D-3C-model-general-limit}
  \left\{ \aligned
  \partial_t\omega_{3} + v_{H} \cdot \nabla \omega_{3} &= \big(\nu \Delta+\mathcal L_{\bar Q} \big)\omega_3 + \nabla\cdot (A\, \omega_{H}),\\
  \partial_t v_3+ v_{H} \cdot \nabla v_3 &= \big(\nu \Delta+\mathcal L_{\bar Q} \big) v_3,
  \endaligned \right.
  \end{equation}
where $\mathcal L_{\bar Q} f = \frac12 \nabla\cdot \big[\bar Q\, \nabla f \big]$ and, as usual, $\omega_3= \nabla^\perp \cdot v_{H}$ and $\omega_{H}= \nabla^\perp v_3$.

Moreover, if the matrices $\bar Q$ and $A$ satisfy
  \begin{equation}\label{eq:cond-unique-limit}
  x\cdot \bar Q x + y\cdot \bar Q y + 2 x\cdot A y^\perp \ge 0 \quad \mbox{for any } x, y\in \R^2,
  \end{equation}
then the system \eqref{eq:2D-3C-model-general-limit} admits a unique weak solution $(\omega_3, v_3)$ in the class $L^\infty(0,T; L^2) \cap L^2(0,T; H^1)$; in this case, the whole family $\big\{ \big(\omega^\ell_3, v^\ell_3\big) \big\}_{\ell\in (0,1)}$ converges in law, in the topology of $ L^2(0,T; L^2)$, to $(\omega_3, v_3)$.
\end{theorem}

This result will be proved in Section \ref{sect-proof-scaling-limit}, where the key ingredient is to prove that the martingale parts vanish in a weak sense. The proof of uniqueness of \eqref{eq:2D-3C-model-general-limit} is standard and follows by applying the Lions-Magenes lemma. As $\bar Q$ is nonnegative definite, it is clear that \eqref{eq:cond-unique-limit} holds in the case $A=0$; in the example given below Remark \ref{rem-sect-5-AKA}, we have $\bar Q= 2\kappa I_2$ and thus \eqref{eq:cond-unique-limit} also holds if $q_0$ defined in \eqref{cond-nonzero-limit-matrix} belongs to the interval $[-1,1]$.

\begin{remark}\label{rem-sect-5-AKA}
Let us interpret the above result. First, Boussinesq hypothesis is verified, but, as expected from the geometric structure of the problem, the additional dissipation acts in the horizontal direction only. Second, the presence of a random stretching term (absent in previous works) may give rise to a first order term, $\nabla\cdot (A\, \omega_H)$, which could increase exponentially the size of the vorticity, like in the dynamo problem of magnetic fields subject to turbulence \cite{Seshasa}.
\end{remark}

Next, we provide an example of vortex noise verifying the above conditions. We begin with heuristic discussions on how to choose the right scaling of noise. Recall \eqref{typical-example} for typical choices of velocity fields of the vortex; we assume that the function $\theta\in C_c^\infty\big((-1/2, 1/2)^2\big)$ is a radially symmetric probability density function, then by Remark \ref{rem-diagonal-horizontal-covar}, the horizontal covariance matrix $Q_H(0)$ is diagonal. In this case, the second order differential operator $\mathcal L$ in It\^o equations \eqref{eq:stoch-2D-3C-model-Ito} takes the form
  $$  \mathcal L= \frac14{\rm Tr}(Q_{H}(0)) \Delta .  $$
As in Remark \ref{rem-estimates-horizontal-covar}, we define rescaled vortex: $\theta_\ell(x)= \ell^{-2} \theta(\ell^{-1}x),\, x\in \T^2$, where $\ell\in (0,1)$ stands for the ``size'' of the vortex; then we define
  $$\sigma_H^\ell= \Gamma_\ell\, K\ast \theta_\ell. $$
We have to determine the dependence of the (horizontal) vortex intensity $\Gamma_\ell$ on $\ell$. By \eqref{eq:trace-lower-bound}, if  $\Gamma_\ell$ were a constant independent of $\ell$, then ${\rm Tr}(Q_H(0)) \sim \|K\ast \theta_\ell \|_{L^2}^2$ explodes at the rate of $\log \ell^{-1}$. In order that the vortex has constant energy as $\ell\to 0$, we choose $\Gamma_\ell$ to be a function of $\ell$ in such a way that $Q_H^\ell(0)$ is a constant matrix.
Thus, we fix some $\kappa>0$ and define
  \begin{equation}\label{eq:Gamma-1-ell}
  \Gamma_\ell= 2\sqrt{\kappa}\, \|K\ast \theta_\ell \|_{L^2}^{-1} \sim (\log \ell^{-1})^{-1/2}.
  \end{equation}
It remains to introduce the rescaled vertical velocity $\sigma_3^\ell$ of the vortex:
  \begin{equation}\label{vertical-velocity}
  \sigma_3^\ell = \gamma_\ell\, G\ast \chi_\ell,
  \end{equation}
where $\gamma_\ell$ vanishes as $\ell\to 0$, $G$ is the Green function on $\T^2$ and $\chi\in C_c^\infty\big((-1/2, 1/2)^2\big)$ is radially symmetric; we remark that $\gamma_\ell$ may not be the same as $\Gamma_\ell$.

Having $\sigma_H^\ell$ and $\sigma_3^\ell$ in mind, we can finally define as above the covariance matrices $Q^\ell, \, Q_H^\ell$ etc., and the associated operators $\mathbb Q^\ell, \, \mathbb Q_H^\ell$. Then by Remark \ref{rem-diagonal-horizontal-covar} and \eqref{eq:Gamma-1-ell}, $Q^\ell_{H}(0) =2 \kappa I_2$ is a constant matrix and thus (a) of Hypothesis \ref{basic-hypotheses} is verified; in particular,
  $$ {\rm Tr}(Q^\ell_{H}(0)) = 4\kappa. $$
Moreover, by Lemma \ref{lem-covar-operator}, one has $\|\mathbb Q_H^\ell \|_{L^2\to L^2} \le \Gamma_\ell^2 \|K \|_{L^1}^2 $ because $\theta$ is a probability density; \eqref{eq:Gamma-1-ell} implies that $\|\mathbb Q_H^\ell \|_{L^2\to L^2}$ vanishes as $\ell\to 0$. In the same way, $\|\mathbb Q_3^\ell \|_{L^2\to L^2} \le \gamma_\ell^2 \|G \|_{L^1}^2 \|\chi \|_{L^1}^2$ vanishes as well in the limit $\ell\to 0$. Therefore, condition (b) in Hypothesis \ref{basic-hypotheses} holds too.

In order to verify conditions (c) and (d) in Hypothesis \ref{basic-hypotheses}, we need some other assumptions.

\begin{proposition} \label{prop:limit-matrix}
Assume that $\theta, \chi\in C_c^\infty\big((-\frac12, \frac12)^2\big)$ are radially symmetric probability density functions.
\begin{itemize}
\item[(1)] If $\gamma_\ell = o(\Gamma_\ell)$ as $\ell\to 0$, then the limit $A:= \lim_{\ell\to 0} \nabla Q_{H,3}^\ell(0)$ is null, i.e. $A=0$.
\item[(2)] Assume that
  \begin{equation} \label{cond-nonzero-limit-matrix}
  q_0 :=\lim_{\ell\to 0} \frac{\gamma_\ell}{ \Gamma_\ell} \neq 0,
  \end{equation}
then the matrix
  $$A = 2\,\kappa\, q_0 \begin{pmatrix}
  0 &\ 1\\
  -1 &\ 0
  \end{pmatrix} . $$
\item[(3)] If $\gamma_\ell = O(\Gamma_\ell)$ as $\ell\to 0$, then $\sup_{\ell\in (0,1)} \|\nabla^2 Q_3^\ell(0) \|_M <+\infty$.
\end{itemize}
\end{proposition}

\begin{remark}
Assume condition \eqref{cond-nonzero-limit-matrix}. From the proof of Lemma \ref{lem-nabla-G} in Section \ref{subsec-proof-prop}, it is clear that we can relax the nonnegativity condition on $\chi$, but assuming $a_0:= \int_{\T^2} \chi \, dx \neq 0$; of course, the constant in $A$ would be different. In this case, one can show as in the proof of Lemma \ref{lem-nabla-G} that
  $$(\partial_i G_{\R^2})\ast \chi(x) \sim a_0\, \partial_i G_{\R^2}(x) \quad \mbox{for } |x|\gg 1, $$
where $G_{\R^2}$ is the Green function on $\R^2$. However, if $a_0=0$, then $(\partial_i G_{\R^2})\ast \chi(x)$ decays too fast as $|x|\to \infty$; this would result in that $A$ is the trivial null matrix.
\end{remark}

\section{A priori estimates and proof of Theorem \ref{thm:well-posedness}} \label{sec-proof-well-posedness}

This section consists of three parts: in Section \ref{subsec-Ito-corrector} we first provide the computations which lead the system \eqref{eq:stoch-2D-3C-model} to its It\^o formulation \eqref{eq:stoch-2D-3C-model-Ito}, then we derive in Section \ref{sec-a-priori} the a priori estimates necessary for proving the existence of solutions, finally we give a sketched proof of Theorem \ref{thm:well-posedness} in Section \ref{subsec-proof-well-posedness}, focusing on the pathwise uniqueness of \eqref{eq:stoch-2D-3C-model-Ito}.

\subsection{It\^o-Stratonovich corrector} \label{subsec-Ito-corrector}

Recall the stochastic 2D-3C model in Stratonovich form:
  \begin{equation}\label{stoch-2D-3C-eq}
   \left\{ \aligned
  d\omega_{3} &= (-v_{H} \cdot \nabla \omega_{3} + \nu\Delta  \omega_3)\, dt - \sum_k \big(\sigma_k^H \cdot \nabla \omega_{3} -\omega_{H}\cdot \nabla  \sigma_k^3 \big)\circ dW^k_t, \\
  dv_3 &= (-v_{H} \cdot \nabla  v_3 + \nu\Delta  v_3)\, dt - \sum_k \sigma_k^H \cdot\nabla v_{3} \circ dW^k_t.
  \endaligned \right.
  \end{equation}
As the equation for $v_3$ is driven by pure transport noise, it is well known that the corresponding It\^o equation is
  $$dv_3 = (-v_{H} \cdot \nabla  v_3 + \nu\Delta  v_3 + \mathcal L v_3)\, dt - \sum_k \sigma_k^H \cdot\nabla v_{3} \, dW^k_t,$$
where the second order differential operator $\mathcal L$ is defined in \eqref{eq:corrector-L}.

We turn to deriving the It\^o equation for $\omega_3$:
  \begin{equation}\label{eq:omega-3}
  \aligned
  d\omega_{3} + (v_{H} \cdot \nabla \omega_{3} - \nu\Delta  \omega_3)\, dt
  & = - \sum_k \big(\sigma_k^H \cdot \nabla \omega_{3} -\omega_{H}\cdot \nabla  \sigma_k^3 \big)\, dW^k_t\\
  &\quad -\frac12 \sum_k  d \big[ \sigma_k^H \cdot \nabla \omega_{3} -\omega_{H}\cdot \nabla  \sigma_k^3, W^k \big]_t .
  \endaligned
  \end{equation}
Noting that $\omega_{H}= \nabla^\perp v_3$, it holds
  $$\aligned
  d \big(\sigma_k^H \cdot \nabla \omega_{3} -\omega_{H}\cdot \nabla  \sigma_k^3 \big)
  &= \sigma_k^H \cdot \nabla (d\omega_{3}) -\nabla ^{\perp}(dv_{3}) \cdot \nabla  \sigma_k^3 \\
  &= V\, dt - \sum_l \sigma_k^H \cdot \nabla  \big(\sigma_l^H \cdot \nabla \omega_{3} -\omega_{H}\cdot \nabla  \sigma_l^3 \big)\circ dW^l_t \\
  &\quad + \sum_l \big(\nabla ^{\perp} (\sigma_l^H \cdot\nabla v_{3}) \cdot \nabla  \sigma_k^3 \big) \circ dW^l_t,
  \endaligned $$
where $V\, dt$ is the finite variation part with
  $$V= \sigma_k^H \cdot \nabla(-v_{H} \cdot \nabla \omega_{3} + \nu\Delta  \omega_3) - \nabla ^{\perp}(-v_{H} \cdot \nabla  v_3 + \nu\Delta  v_3) \cdot \nabla  \sigma_k^3,$$
thus
  $$\aligned
  & d \big[ \sigma_k^H \cdot \nabla \omega_{3} -\omega_{H}\cdot \nabla  \sigma_k^3, W^k \big]_t\\
  &= -\sigma_k^H \cdot \nabla  \big(\sigma_k^H \cdot \nabla \omega_{3} -\omega_{H}\cdot \nabla  \sigma_k^3 \big)\, dt + \big(\nabla ^{\perp} (\sigma_k^H \cdot\nabla v_{3}) \cdot \nabla  \sigma_k^3 \big)\, dt.
  \endaligned $$
Therefore, the last term in \eqref{eq:omega-3} (modulo $dt$) is
  \begin{equation}\label{eq:Ito-corrector}
  \aligned
  & \frac12 \sum_k \sigma_k^H \cdot \nabla  \big(\sigma_k^H \cdot \nabla \omega_{3} -\omega_{H}\cdot \nabla  \sigma_k^3 \big) -\frac12 \sum_k \nabla ^{\perp} (\sigma_k^H \cdot\nabla v_{3}) \cdot \nabla  \sigma_k^3 \\
  &= \frac12 \sum_k \sigma_k^H \cdot \nabla  (\sigma_k^H \cdot \nabla \omega_{3}) - \frac12 \sum_k \sigma_k^H \cdot \nabla  \big( \omega_{H}\cdot \nabla  \sigma_k^3 \big)\\
  &\quad - \frac12 \sum_k \nabla^{\perp} (\sigma_k^H \cdot\nabla v_{3}) \cdot \nabla  \sigma_k^3 \\
  &=: \mathcal L \omega_3 -\frac12 (I+J).
  \endaligned
  \end{equation}
It is sufficient to compute the last two sums, denoted by $I$ and $J$ respectively.

\begin{lemma}\label{lem-I-J-expression}
One has
  $$I=J= \sum_k (\sigma_k^H \cdot \nabla  \omega_{H} ) \cdot \nabla  \sigma_k^3 -\omega_{H} \cdot \begin{pmatrix}
  \sum_k ( \partial_1 \sigma_k^H) \cdot\nabla  \sigma_k^3 \smallskip \\
  \sum_k ( \partial_2 \sigma_k^H) \cdot\nabla  \sigma_k^3
  \end{pmatrix} . $$
\end{lemma}

\begin{proof}
We have
  $$\aligned
  I&= \sum_k \sigma_k^H \cdot \nabla  \big( \omega_{H}\cdot \nabla  \sigma_k^3 \big) \\
  &= \sum_k (\sigma_k^H \cdot \nabla  \omega_{H} ) \cdot \nabla  \sigma_k^3 + \omega_{H} \cdot\bigg( \sum_k \sigma_k^H \cdot \nabla  ( \nabla  \sigma_k^3 ) \bigg) .
  \endaligned $$
The second term can be slightly simplified: for $i\in \{1,2\}$,
  $$\aligned
  \sum_k \sigma_k^H \cdot \nabla  ( \partial_i \sigma_k^3 )
  &= \sum_k \sigma_k^H \cdot \partial_i (\nabla  \sigma_k^3 ) \\
  &= \sum_k \big[ \partial_i ( \sigma_k^H \cdot\nabla  \sigma_k^3 ) - ( \partial_i \sigma_k^H) \cdot\nabla  \sigma_k^3 \big] \\
  &= \partial_i \bigg[\sum_k \sigma_k^H \cdot\nabla  \sigma_k^3 \bigg]- \sum_k ( \partial_i \sigma_k^H) \cdot\nabla  \sigma_k^3 \\
  &= - \sum_k ( \partial_i \sigma_k^H) \cdot\nabla  \sigma_k^3,
  \endaligned $$
where we have used the fact that
  $$\sum_k \sigma_k^H \cdot\nabla \sigma_k^3= \sum_{j=1}^2 \sum_k \sigma_k^j\, \partial_j \sigma_k^3= -\sum_{j=1}^2 \partial_j Q_{j,3}(0)$$
is a constant matrix, due to the space homogeneity of covariance function. From this we immediately get the expression for $I$.

It remains to compute
  $$\aligned
  J&= \sum_k \nabla ^{\perp} (\sigma_k^H \cdot\nabla v_{3}) \cdot \nabla  \sigma_k^3\\
  &= \sum_k \nabla ^{\perp}(\sigma_k^1 \partial_1 v_3 + \sigma_k^2 \partial_2 v_3) \cdot \nabla  \sigma_k^3 \\
  &= \sum_k \big[ (\partial_1 v_3) \nabla ^{\perp} \sigma_k^1 + (\partial_2 v_3)\nabla ^{\perp} \sigma_k^2 \big]  \cdot \nabla  \sigma_k^3 \\
  &\quad + \sum_k \big[ \sigma_k^1 \nabla ^{\perp} (\partial_1 v_3) + \sigma_k^2 \nabla ^{\perp} (\partial_2 v_3) \big] \cdot \nabla  \sigma_k^3\, ;
  \endaligned $$
note that $\nabla ^{\perp} (\partial_i v_3)= \partial_i \nabla ^{\perp} v_3= \partial_i \omega_{H}$, thus,
  $$\aligned
  J&= (\partial_1 v_3) \sum_k \nabla ^{\perp} \sigma_k^1  \cdot \nabla  \sigma_k^3 + (\partial_2 v_3) \sum _k \nabla ^{\perp} \sigma_k^2 \cdot \nabla  \sigma_k^3 \\
  &\quad + \sum_k \big[ \sigma_k^1 \partial_1 \omega_{H} + \sigma_k^2 \partial_2 \omega_{H} \big] \cdot \nabla  \sigma_k^3 \\
  &= \omega_{H} \cdot \begin{pmatrix}
  \sum _k \nabla ^{\perp} \sigma_k^2 \cdot \nabla  \sigma_k^3 \smallskip \\
  - \sum_k \nabla ^{\perp} \sigma_k^1  \cdot \nabla  \sigma_k^3
  \end{pmatrix} + \sum_k (\sigma_k^H \cdot \nabla\omega_{H}) \cdot \nabla  \sigma_k^3.
  \endaligned $$

We see that $I$ and $J$ contain a common term; in order to show that they coincide, it suffices to check that the remaining terms are the same. Indeed,
  $$\partial_1 \sigma_k^H= \begin{pmatrix}
  \partial_1 \sigma_k^1 \smallskip \\
  \partial_1 \sigma_k^2 \end{pmatrix}
  = \begin{pmatrix}
  -\partial_2 \sigma_k^2 \smallskip \\
  \partial_1 \sigma_k^2 \end{pmatrix}
  = -\nabla ^\perp \sigma_k^2.  $$
In the same way, $\partial_2 \sigma_k^H= \nabla ^\perp \sigma_k^1$; as a consequence, $I=J$.
\end{proof} \medskip

Now we show that the second addend in $I$ and $J$ vanishes.

\begin{lemma}\label{lem-I-J-expression-2}
It holds that
  \begin{equation}\label{eq:I=J}
  I=J= \sum_k (\sigma_k^H \cdot \nabla\omega_{H}) \cdot \nabla  \sigma_k^3.
  \end{equation}
\end{lemma}

\begin{proof}
First of all, we prove a simple fact: let $Q=(Q_{i,j})_{1\le i,j \le 3}$ be the covariance matrix, then for any $i,j\in \{1,2,3\}$, $m,n \in \{1,2\}$, it holds
  \begin{equation}\label{eq:covariance-derivative}
  -\partial_m \partial_n Q_{i,j}(0) = \sum_k \partial_m \sigma_k^i(x)\, \partial_n \sigma_k^j(x)= \sum_k \partial_n \sigma_k^i(x)\, \partial_m \sigma_k^j(x).
  \end{equation}
Indeed, recalling that for $i,j  \in \{1,2, 3\}$,
  $$ Q_{i,j}(x-y)= \sum_k \sigma_k^i(x) \, \sigma_k^j(y), \quad x,y\in \T^2, $$
thus we have
  $$\partial_{x_m} \partial_{y_n} [Q_{i,j}(x-y)]= \sum_k \partial_{x_m} \sigma_k^i(x) \, \partial_{y_n} \sigma_k^j(y);$$
that is,
  $$-\partial_m \partial_n Q_{i,j}(x-y)= \sum_k \partial_m \sigma_k^i(x)\, \partial_n \sigma_k^j(y). $$
Letting $x=y$ gives the first identity. For the second one, we have
  $$\partial_{x_n} \partial_{y_m} [Q_{i,j}(x-y)]= \sum_k \partial_{x_n} \sigma_k^i(x) \, \partial_{y_m} \sigma_k^j(y);$$
from here we get the equality in the same way as above.

Now we start proving equality \eqref{eq:I=J}; recall that
  $$I= \sum_k (\sigma_k^H \cdot \nabla  \omega_{H} ) \cdot \nabla  \sigma_k^3 + \omega_{H} \cdot \begin{pmatrix}
  -\sum_k ( \partial_1 \sigma_k^H) \cdot\nabla  \sigma_k^3 \smallskip \\
  -\sum_k ( \partial_2 \sigma_k^H) \cdot\nabla  \sigma_k^3
  \end{pmatrix} ,$$
we have
  $$\aligned
  \sum_k ( \partial_1 \sigma_k^H) \cdot\nabla  \sigma_k^3 &= \sum_k \big(\partial_1 \sigma_k^1\, \partial_1 \sigma_k^3 + \partial_1 \sigma_k^2\, \partial_2 \sigma_k^3 \big)\\
  &= \sum_k \big(-\partial_2 \sigma_k^2\, \partial_1 \sigma_k^3 + \partial_1 \sigma_k^2\, \partial_2 \sigma_k^3 \big),
  \endaligned $$
where we have used $\partial_1 \sigma_k^1 + \partial_2 \sigma_k^2=0$. By the above identity \eqref{eq:covariance-derivative},
  $$\aligned
  \sum_k ( \partial_1 \sigma_k^H) \cdot\nabla  \sigma_k^3 &= \partial_1 \partial_2 Q_{2,3}(0) -  \partial_1 \partial_2 Q_{2,3}(0) =0.
  \endaligned$$
In the same way,
  $$\aligned
  \sum_k ( \partial_2 \sigma_k^H) \cdot\nabla  \sigma_k^3 &= \sum_k \big(\partial_2 \sigma_k^1\, \partial_1 \sigma_k^3 + \partial_2 \sigma_k^2\, \partial_2 \sigma_k^3 \big) \\
  &= \sum_k \big(\partial_2 \sigma_k^1\, \partial_1 \sigma_k^3 -\partial_1 \sigma_k^1\, \partial_2 \sigma_k^3 \big) \\
  &= -\partial_1 \partial_2 Q_{1,3}(0) + \partial_1 \partial_2 Q_{1,3}(0) =0.
  \endaligned $$
This implies \eqref{eq:I=J} and completes the proof of Lemma \ref{lem-I-J-expression-2}.
\end{proof}

Recall the constant matrix $\nabla Q_{H,3}(0)$ defined in \eqref{eq:nabla-Q-H-3}; we can further simplify $I=J$ as follows.

\begin{proposition}\label{prop-I-J-expression}
Let $Q_{H,3}(x)= \big(Q_{1,3}(x), Q_{2,3}(x)\big)^\ast\, (x\in \T^2)$ be a vector valued function. Then,
  $$I=J= - \nabla \cdot [\nabla Q_{H,3}(0)\, \omega_{H}].$$
\end{proposition}

\begin{proof}
Recalling equality \eqref{eq:I=J}, we have
  $$\aligned
  I&=\sum_k \big[(\partial_1 \sigma_k^3)\, \sigma_k^H \cdot\nabla \omega_1 + (\partial_2 \sigma_k^3)\, \sigma_k^H \cdot\nabla \omega_2\big] \\
  &=\sum_k \begin{pmatrix}
  \sigma_k^1\partial_1 \sigma_k^3 \smallskip \\
  \sigma_k^2\partial_1 \sigma_k^3 \end{pmatrix} \cdot\nabla \omega_1
  + \sum_k \begin{pmatrix}
  \sigma_k^1\partial_2 \sigma_k^3 \smallskip \\
  \sigma_k^2\partial_2 \sigma_k^3 \end{pmatrix} \cdot\nabla \omega_2 \\
  &= -\begin{pmatrix}
  \partial_1 Q_{1,3}(0) \\ \partial_1 Q_{2,3}(0)
  \end{pmatrix} \cdot\nabla \omega_1
  - \begin{pmatrix}
  \partial_2 Q_{1,3}(0) \\ \partial_2 Q_{2,3}(0)
  \end{pmatrix} \cdot\nabla \omega_2 \\
  &= - \partial_1 Q_{H,3}(0) \cdot\nabla \omega_1 - \partial_2 Q_{H,3}(0) \cdot\nabla \omega_2.
  \endaligned $$
Then,
  $$\aligned
  I&= -\nabla \cdot\big[ \omega_1\, \partial_1 Q_{H,3}(0) \big] -\nabla \cdot\big[ \omega_2\, \partial_2 Q_{H,3}(0) \big] = - \nabla \cdot [\nabla Q_{H,3}(0)\, \omega_{H}],
  \endaligned $$
which finishes the proof.
\end{proof} \medskip

Combining \eqref{eq:omega-3}, \eqref{eq:Ito-corrector} and \eqref{eq:I=J}, we obtain
  \begin{equation}\label{eq:Ito-omega-3}
  \aligned
  d\omega_{3} + (v_{H} \cdot \nabla \omega_{3} - \nu\Delta  \omega_3)\, dt
  & = - \sum_k \big(\sigma_k^H \cdot \nabla \omega_{3} -\omega_{H}\cdot \nabla  \sigma_k^3 \big)\, dW^k_t\\
  &\quad + \big(\mathcal L \omega_3 + \nabla \cdot [\nabla Q_{H,3}(0)\, \omega_{H}] \big)\, dt,
  \endaligned
  \end{equation}
where $\nabla Q_{H,3} (0)$ and the operator $\mathcal L $ are defined, respectively, in \eqref{eq:nabla-Q-H-3} and \eqref{eq:corrector-L}.

\subsection{A priori estimates}\label{sec-a-priori}

This section is devoted to proving a priori estimates for the stochastic 2D-3C model \eqref{stoch-2D-3C-eq} (i.e. \eqref{eq:stoch-2D-3C-model}). First, as $v_3$ satisfies a stochastic advection-diffusion equation with divergence free fields $v_H$ and $\sigma_k^H$, one has $\P$-a.s.,
  \begin{equation}\label{eq:a-priori-v-3}
  \|v_3(t)\|_{L^2}^2 + 2\nu \int_0^t \|\nabla v_3(s)\|_{L^2}^2\, ds \le \|v_3^0 \|_{L^2}^2 \quad \mbox{for all } t\ge 0.
  \end{equation}

Next, we turn to dealing with the estimate on $\omega_3$; recall that $\|\cdot \|_M$ is a matrix norm.

\begin{lemma}\label{lem-a-priori-1}
We have
  \begin{equation}\label{lem-a-priori-1.1}
  \E\bigg[\sup_{t\in [0,T]} \|\omega_3(t)\|_{L^2}^2 + \nu \int_0^T \|\nabla \omega_3(s)\|_{L^2}^2\, ds \bigg] \le C_{\nu, Q} \big(\|v_3^0 \|_{L^2}^2 + \|\omega_3^0 \|_{L^2}^2\big),
  \end{equation}
where $C_{\nu, Q}$ is a constant depending on $\nu, \|\nabla Q_{H,3}(0)\|_M, \|\nabla^2 Q_3(0)\|_M$.
\end{lemma}

\begin{proof}
In the sequel we will write $\nabla^\perp v_3$ instead of $\omega_{H}$, in order to make use of the equation for $v_3$. By the first equation in \eqref{stoch-2D-3C-eq}, we have
  $$\aligned
  d \left\Vert \omega_{3}\right\Vert _{L^{2}}^{2}+2\nu\left\Vert \nabla\omega_{3}\right\Vert _{L^{2}}^{2} dt
  &= -2\sum_{k} \< \nabla \omega_{3} \cdot\nabla^\perp v_3, \sigma_k^3 \> \circ dW^k_t.
  \endaligned$$
Transforming in It\^o differential yields
  \begin{equation}\label{eq:Ito}
  \aligned
  d \left\Vert \omega_{3}\right\Vert _{L^{2}}^{2}+2\nu\left\Vert \nabla\omega_{3}\right\Vert _{L^{2}}^{2} dt
  &= -2\sum_k \< \nabla \omega_{3} \cdot\nabla^{\perp}v_{3}, \sigma_k^3 \> \, dW^k_t \\
  &\quad - \sum_k d\big[\< \nabla \omega_{3} \cdot\nabla^{\perp}v_{3}, \sigma_k^3 \>, W^k \big]_t.
  \endaligned
  \end{equation}
We have
  $$d\< \nabla\omega_{3} \cdot\nabla^{\perp}v_{3}, \sigma_k^3 \> = \big\< \nabla(d\omega_{3} ) \cdot\nabla^{\perp}v_{3}, \sigma_k^3 \big\> + \big\< \nabla \omega_{3} \cdot\nabla^{\perp}(dv_{3}), \sigma_k^3 \big\>, $$
and, using the equations in \eqref{stoch-2D-3C-eq} for $\omega_{3}$ and $v_3$,
$$\aligned
  d\< \nabla \omega_{3} \cdot\nabla^{\perp}v_{3}, \sigma_k^3 \> &= \tilde V\, dt - \sum_l \big\< \nabla (\sigma_l^H \cdot \nabla \omega_{3} ) \cdot\nabla^{\perp}v_{3}, \sigma_k^3 \big\>\circ dW^l_t \\
  &\quad + \sum_l \big\< \nabla (\nabla^{\perp}v_{3}\cdot \nabla \sigma_l^3 ) \cdot\nabla^{\perp}v_{3}, \sigma_k^3 \big\>\circ dW^l_t \\
  &\quad - \sum_l \big\< \nabla \omega_{3} \cdot\nabla^{\perp}(\sigma_l^H \cdot\nabla v_{3} ), \sigma_k^3 \big\> \circ dW^l_t ,
  \endaligned $$
where $\tilde V\, dt$ is the finite variation part; as a result,
  $$\aligned
  -d\big[\< \nabla \omega_{3} \cdot\nabla^{\perp}v_{3}, \sigma_k^3 \>, W^k \big]_t &= \big\< \nabla (\sigma_k^H \cdot \nabla \omega_{3} ) \cdot\nabla^{\perp}v_{3}, \sigma_k^3 \big\> \, dt \\
  &\quad -\big\< \nabla (\nabla^{\perp}v_{3}\cdot \nabla \sigma_k^3 ) \cdot\nabla^{\perp}v_{3}, \sigma_k^3 \big\> \,d t \\
  &\quad +\big\< \nabla \omega_{3} \cdot\nabla^{\perp}(\sigma_k^H \cdot\nabla v_{3} ), \sigma_k^3 \big\>\,d t.
  \endaligned $$
Integration by parts:
  \begin{equation}\label{eq:quadratic-varia}
  \aligned
  -d\big[\< \nabla \omega_{3} \cdot\nabla^{\perp}v_{3}, \sigma_k^3 \>, W^k \big]_t &=  -\big\< (\sigma_k^H \cdot \nabla \omega_{3} ) \nabla^{\perp}v_{3}, \nabla \sigma_k^3 \big\> \, dt \\
  &\quad + \big\<(\nabla^{\perp}v_{3}\cdot \nabla \sigma_k^3 ) \nabla^{\perp}v_{3},  \nabla \sigma_k^3 \big\> \,d t \\
  &\quad - \big\< (\nabla \omega_{3})  (\sigma_k^H \cdot\nabla v_{3} ), \nabla^{\perp} \sigma_k^3 \big\>\,d t.
  \endaligned
  \end{equation}

Let us consider the first term on the right-hand side:
  $$\aligned
  \big\< (\sigma_k^H \cdot \nabla \omega_{3} ) \nabla^{\perp}v_{3}, \nabla \sigma_k^3 \big\>
  &= \int (\sigma_k^H \cdot \nabla \omega_{3} ) (\nabla^{\perp}v_{3}\cdot \nabla \sigma_k^3)\, d x \\
  &= \int (\nabla \omega_{3})^\ast \sigma_k^H (\nabla \sigma_k^3)^\ast \nabla^{\perp}v_{3} \, d x .
  \endaligned $$
We have
  \begin{equation}\label{eq:matrix-identity-1}
  \aligned
  \sum_k \sigma_k^H (\nabla \sigma_k^3)^\ast &= \sum_k \begin{pmatrix}
  \sigma_k^1 \partial_1 \sigma_k^3 &\ \sigma_k^1 \partial_2 \sigma_k^3 \smallskip \\
  \sigma_k^2 \partial_1 \sigma_k^3 &\ \sigma_k^2 \partial_2 \sigma_k^3
  \end{pmatrix}
  = - \begin{pmatrix}
  \partial_1 Q_{1,3}(0) &\ \partial_2 Q_{1,3}(0)\smallskip \\
  \partial_1 Q_{2,3}(0) &\ \partial_2 Q_{2,3}(0)
  \end{pmatrix}
  = -\nabla Q_{H,3}(0);
  \endaligned
  \end{equation}
as a result,
  $$\aligned
  - \sum_k \big\< (\sigma_k^H \cdot \nabla \omega_{3} ) \nabla_{H}^{\perp}v_{3}, \nabla \sigma_k^3 \big\>
  &= \int (\nabla \omega_{3})^\ast \nabla Q_{H,3}(0) \nabla^{\perp}v_{3} \, d x \\
  &\le \|\nabla Q_{H,3}(0)\|_M \|\nabla \omega_{3} \|_{L^2} \|\nabla v_{3} \|_{L^2} \\
  &\le \frac\nu 2 \|\nabla \omega_{3} \|_{L^2}^2 + \frac{\|\nabla Q_{H,3}(0)\|_M^2}{2\nu} \|\nabla v_{3} \|_{L^2}^2.
  \endaligned $$
The third term in \eqref{eq:quadratic-varia} can be estimated in the same way, thus
  \begin{equation}\label{eq:estimate-1-3}
  \aligned
  & -\sum_k \big\< (\sigma_k^H \cdot \nabla \omega_{3} ) \nabla^{\perp}v_{3}, \nabla \sigma_k^3 \big\> - \sum_k \big\< (\nabla \omega_{3})  (\sigma_k^H \cdot\nabla v_{3} ), \nabla^{\perp} \sigma_k^3 \big\> \\
  &\quad \le \nu \|\nabla \omega_{3} \|_{L^2}^2 + \nu^{-1} \|\nabla Q_{H,3}(0)\|_M^2 \|\nabla v_{3} \|_{L^2}^2.
  \endaligned
  \end{equation}

Next, for the second term in \eqref{eq:quadratic-varia}, we have
  \begin{equation} \label{eq:useful-identity}
  \sum_k \big\<(\nabla^{\perp}v_{3}\cdot \nabla \sigma_k^3 ) \nabla^{\perp}v_{3},  \nabla \sigma_k^3 \big\> = \sum_k \big\| \nabla^{\perp}v_{3}\cdot \nabla \sigma_k^3 \big\|_{L^2}^2.
  \end{equation}
Recall that
  $$Q_3(x-y)= Q_{3,3}(x-y) = \sum_k \sigma_k^3(x)\, \sigma_k^3(y),\quad x,y\in \T^2;$$
by the computations at the beginning of the proof of Lemma \ref{lem-I-J-expression-2}, it holds
  $$- \partial_j \partial_i Q_3(0)= \sum_k \partial_i\sigma_k^3(x)\, \partial_j\sigma_k^3(x).$$
In matrix form, it reads as
  \begin{equation}\label{eq:matrix-identity-2}
  -\nabla^2 Q_3(0)= \sum_k \nabla \sigma_k^3(x)\, (\nabla \sigma_k^3(x))^\ast.
  \end{equation}
Therefore,
  $$\aligned
  \sum_k \big\| \nabla^{\perp}v_{3}\cdot \nabla \sigma_k^3 \big\|_{L^2}^2 &=  \sum_k \int \big( \nabla^{\perp}v_{3}\cdot \nabla \sigma_k^3 \big)^2\, dx \\
  &= \sum_k \int (\nabla^{\perp}v_{3})^\ast\, (\nabla \sigma_k^3) \,(\nabla \sigma_k^3)^\ast (\nabla^{\perp}v_{3}) \, dx \\
  &= - \int (\nabla^{\perp}v_{3})^\ast \, \nabla^2 Q_3(0)\, (\nabla^{\perp}v_{3}) \, dx \\
  &= - \big\<\nabla^{\perp}v_{3}, \nabla^2 Q_3(0)\, (\nabla^{\perp}v_{3}) \big\>.
  \endaligned $$
As a result, by \eqref{eq:useful-identity},
  \begin{equation} \label{eq:useful-identity.1}
  \aligned
  \sum_k \big\<(\nabla^{\perp}v_{3}\cdot \nabla \sigma_k^3 ) \nabla^{\perp}v_{3},  \nabla \sigma_k^3 \big\>
  &= - \big\<\nabla^{\perp}v_{3}, \nabla^2 Q_3(0)\, (\nabla^{\perp}v_{3}) \big\> \\
  &\le \|\nabla^2 Q_3(0)\|_M \|\nabla v_{3}\|_{L^2}^2.
  \endaligned
  \end{equation}

Combining the above estimate with \eqref{eq:quadratic-varia} and \eqref{eq:estimate-1-3}, we arrive at
  $$\aligned
  & - \sum_k d\big[\< \nabla \omega_{3} \cdot\nabla^{\perp}v_{3}, \sigma_k^3 \>, W^k \big]_t \le \nu \|\nabla \omega_{3} \|_{L^2}^2\, dt + \tilde C_{\nu,Q} \|\nabla v_{3}\|_{L^2}^2\,d t ,
  \endaligned $$
where $\tilde C_{\nu,Q}$ is a constant defined as
  $$\tilde C_{\nu,Q} = \nu^{-1} \|\nabla Q_{H,3}(0)\|_M^2 + \|\nabla^2 Q_3(0)\|_M.$$
Substituting this estimate into \eqref{eq:Ito} and noticing that $\< \nabla \omega_{3} \cdot\nabla^{\perp}v_{3},\sigma_k^3 \> = - \< \omega_{3},\nabla^{\perp}v_{3} \cdot \nabla \sigma_k^3 \>$, we arrive at
  \begin{equation}\label{eq:a-priori-omega-3}
  \aligned
  d \| \omega_3 \|_{L^2}^2+ \nu\| \nabla\omega_3 \|_{L^2}^2\, dt
  & \le 2\sum_k \< \omega_{3},\nabla^{\perp} v_3 \cdot \nabla \sigma_k^3 \> \, dW^k_t + \tilde C_{\nu,Q} \|\nabla v_{3}\|_{L^2}^2 \, dt.
  \endaligned
  \end{equation}
Let $M(t)$ be the martingale part and define the stopping times $\tau_n = \inf\{t\ge 0: \| \omega_3(t) \|_{L^2} \ge n\},\, n\ge 1$; then
  $$\aligned
  \E \bigg[ \sup_{t\in [0,T]} \| \omega_3(t\wedge \tau_n) \|_{L^2}^2 \bigg] & \le \| \omega_3^0 \|_{L^2}^2 + \E\bigg[ \sup_{t\in [0,T]} | M(t\wedge \tau_n)| \bigg] + \tilde C_{\nu,Q} \int_0^T \|\nabla v_3(t) \|_{L^2}^2 \, dt \\
  &\le \| \omega_3^0 \|_{L^2}^2 + \E\bigg[ \sup_{t\in [0,T]} | M(t\wedge \tau_n)| \bigg] +\tilde C_{\nu,Q}\, \nu^{-1} \| v_3^0 \|_{L^2}^2,
  \endaligned $$
where the second step is due to \eqref{eq:a-priori-v-3}. We have
  $$\aligned \E\bigg[ \sup_{t\in [0,T]} | M(t\wedge \tau_n)| \bigg]
  &\lesssim \E \bigg[\Big( \sum_k \int_0^{T\wedge \tau_n} \big\<\omega_3(t), \nabla^{\perp} v_3(t) \cdot\nabla \sigma_k^3 \big\>^2 \, dt \Big)^{1/2} \bigg] \\
  &\le \E \bigg[\Big( \int_0^{T\wedge \tau_n} \| \omega_3(t)\|_{L^2}^2 \sum_k \big\| \nabla^{\perp} v_3(t) \cdot\nabla \sigma_k^3 \big\|_{L^2}^2 \,dt \Big)^{1/2} \bigg] \\
  &\le C_Q \, \E \bigg[\Big( \int_0^{T\wedge \tau_n} \| \omega_3(t)\|_{L^2}^2 \| \nabla v_3 (t) \|_{L^2}^2 \,dt \Big)^{1/2} \bigg],
  \endaligned $$
where the last step follows from \eqref{eq:useful-identity.1} and $C_Q= \|\nabla^2 Q_3(0)\|_M^{1/2}$. Therefore, by \eqref{eq:a-priori-v-3} and Cauchy's inequality,
  $$\aligned \E\bigg[ \sup_{t\in [0,T]} | M(t\wedge \tau_n)| \bigg]
  &\le C_Q \, \E \bigg[\sup_{t\in [0,T]} \| \omega_3(t\wedge \tau_n)\|_{L^2} \Big( \int_0^{T} \| \nabla v_3 (t) \|_{L^2}^2 \,dt \Big)^{1/2} \bigg] \\
  &\le C_Q\, \E \bigg[ \nu^{-1/2} \| v_3^0 \|_{L^2} \sup_{t\in [0,T]} \| \omega_3(t\wedge \tau_n)\|_{L^2} \bigg] \\
  &\le \frac12\, \E \bigg[ \sup_{t\in [0,T]} \| \omega_3(t\wedge \tau_n)\|_{L^2}^2 \bigg]+ \frac{ C_Q^2}{2\nu} \| v_3^0 \|_{L^2}^2.
  \endaligned$$
Combining the above estimates, we obtain
  $$\aligned
  \E \bigg[ \sup_{t\in [0,T]} \| \omega_3(t\wedge \tau_n) \|_{L^2}^2 \bigg] &\le \hat C_{\nu,Q} \big(\| v_3^0 \|_{L^2}^2+ \| \omega_3^0 \|_{L^2}^2\big).
  \endaligned $$
By Fatou's lemma, letting $n\to\infty$ yields the first estimate. The estimate on $\E \int_0^T \| \nabla\omega_3(t) \|_{L^2}^{2} dt$ follows easily from \eqref{eq:a-priori-omega-3}.
\end{proof}

In order to apply the compactness method for proving existence of weak solutions, we need also to establish the following estimates.

\begin{lemma}\label{lem-a-priori-2}
Fix some $\alpha\in (0,1/2)$. It holds that
  $$\E\bigg[\int_0^T\!\! \int_0^T \frac{\|v_3(t) -v_3(s)\|_{H^{-2}}^2}{|t-s|^{1+2\alpha}} \,dtds + \int_0^T\!\! \int_0^T \frac{\|\omega_3(t) -\omega_3(s)\|_{H^{-2}}^2}{|t-s|^{1+2\alpha}} \,dtds \bigg] \le \tilde C<+\infty. $$
\end{lemma}

\begin{proof}
Using the equation for $v_3$, it is not difficult to show that
  \begin{equation}\label{lem-a-priori-2.1}
  \E\big(\|v_3(t) -v_3(s)\|_{H^{-1}}^2 \big) \le C_{\nu, Q, T} (t-s) \big(1+ \| \omega_3^0 \|_{L^2}^4+ \| v_3^0 \|_{L^2}^4\big).
  \end{equation}
We omit the proof here since they are easier than that of the estimate on $\omega_3$ given below.

By the equation \eqref{eq:Ito-omega-3} for $\omega_3$, we have
  $$\aligned
  \omega_3(t) -\omega_3(s) &= -\int_s^t (v_H \cdot \nabla \omega_3)(r)\, dr + \int_s^t \big[\nu\Delta \omega_3 + \mathcal L \omega_3 + \nabla \cdot (\nabla Q_{H,3}(0)\, \omega_H)\big](r)\, dr \\
  &\quad -\sum_k\int_s^t \big(\sigma_k^H \cdot \nabla \omega_3 -\omega_H\cdot \nabla\sigma_k^3 \big)(r)\, dW^k_r;
  \endaligned $$
therefore,
  $$\aligned
  \|\omega_3(t) -\omega_3(s)\|_{H^{-2}}^2 &\lesssim I_1+ I_2 +I_3,
  \endaligned $$
where
  $$\aligned
  I_1 &\le (t-s) \int_s^t \|(v_H \cdot \nabla \omega_3)(r) \|_{H^{-2}}^2 \, dr , \\
  I_2 &\le (t-s) \int_s^t \big\|\big[\nu\Delta \omega_3 + \mathcal L \omega_3 + \nabla \cdot (\nabla Q_{H,3}(0)\, \omega_H)\big](r) \big\|_{H^{-2}}^2 \, dr , \\
  I_3 &\le \bigg\| \sum_k\int_s^t \big(\sigma_k^H \cdot \nabla \omega_3 -\omega_H\cdot \nabla\sigma_k^3 \big)(r)\, dW^k_r \bigg\|_{H^{-2}}^2.
  \endaligned $$

First, note that $v_H$ is a divergence free vector field on $\T^2$, it holds
  $$\|(v_H \cdot \nabla \omega_3)(r) \|_{H^{-2}} \lesssim \|v_H(r)\, \omega_3(r)\|_{H^{-1}} \lesssim \|v_H(r)\, \omega_3(r)\|_{L^{3/2}} \le \|v_H(r) \|_{L^6} \|\omega_3(r)\|_{L^2}, $$
where in the last two steps we have used the embedding $L^{3/2} \hookrightarrow H^{-1}$ and H\"older's inequality. Using the Sobolev embedding $H^1 \hookrightarrow L^6$, we arrive at
  $$\|(v_H \cdot \nabla \omega_3)(r) \|_{H^{-2}} \lesssim \|v_H(r) \|_{H^1} \|\omega_3(r)\|_{L^2} \lesssim \|\omega_3(r)\|_{L^2}^2 $$
since $v_H(r) =K\ast \omega_3(r)$ where $K$ is the Biot-Savart kernel on $\T^2$. Recalling the definition of $I_1$, we arrive at
  $$\E I_1 \lesssim (t-s) \int_s^t \E\big(\|\omega_3(r)\|_{L^2}^4 \big) \, dr \lesssim_{\nu, Q, T} (t-s)^2 \big( \|\omega_3^0 \|_{L^2}^4 + \|v_3^0 \|_{L^2}^4 \big), $$
where in the second step we have used Lemma \ref{lem-a-priori-fourth-order} at the end of this section.

Next, for $I_2$, recalling that $\mathcal L\omega_3= \frac12 \nabla \cdot \big[Q_H(0) \nabla \omega_3 \big]$  (see \eqref{eq:corrector-L}) and $\omega_H =\nabla^\perp v_3$, we have
  $$\aligned
  I_2 &\lesssim (t-s) \int_s^t \Big[ \nu^2 \|\omega_3(r) \|_{L^2}^2 + \|Q_H(0) \nabla \omega_3(r) \|_{H^{-1}}^2 + \big\| \nabla Q_{H,3}(0) \nabla^\perp v_3(r) \big\|_{H^{-1}}^2 \Big] \, dr \\
  &\lesssim_{\nu, Q} (t-s) \int_s^t \big[ \|\omega_3(r) \|_{L^2}^2 + \|v_3(r) \|_{L^2}^2 \big] \, dr ;
  \endaligned $$
therefore, by Lemma \ref{lem-a-priori-1} and \eqref{eq:a-priori-v-3}, it holds
  $$\E I_2 \lesssim_{\nu, Q, T} (t-s)^2 \big( \|\omega_3^0 \|_{L^2}^2 + \|v_3^0 \|_{L^2}^2 \big).  $$

Finally, we have
  $$\aligned
  \E I_3 &\lesssim \E \bigg[ \sum_k\int_s^t \big\| \sigma_k^H \cdot \nabla \omega_3(r) -\nabla^\perp v_3(r)\cdot \nabla\sigma_k^3 \big\|_{H^{-2}}^2\, d r \bigg] \\
  &\lesssim \E \bigg[ \sum_k\int_s^t \Big( \big\| \sigma_k^H \cdot \nabla \omega_3(r) \big\|_{H^{-2}}^2 +\big\| \nabla^\perp v_3(r)\cdot \nabla\sigma_k^3 \big\|_{H^{-2}}^2 \Big)\, d r \bigg].
  \endaligned $$
As $\{\sigma_k^H \}_k$ are divergence free, it holds
  $$\sum_k \big\| \sigma_k^H \cdot \nabla \omega_3(r) \big\|_{H^{-2}}^2 \lesssim \sum_k \big\| \sigma_k^H \omega_3(r) \big\|_{H^{-1}}^2 \lesssim \sum_k \big\| \sigma_k^H \omega_3(r) \big\|_{L^2}^2 \le {\rm Tr}(Q_H(0)) \| \omega_3(r) \|_{L^2}^2; $$
moreover,
  $$\aligned
  \sum_k \big\| \nabla^\perp v_3(r)\cdot \nabla\sigma_k^3 \big\|_{H^{-2}}^2
  &= \sum_k \big\| \nabla^\perp \cdot \big( v_3(r) \nabla\sigma_k^3 \big) \big\|_{H^{-2}}^2 \lesssim \sum_k \big\| v_3(r) \nabla\sigma_k^3 \big\|_{L^2}^2 \\
  &\le \|\nabla^2 Q_3(0)\|_M \|v_3(r) \|_{L^2}^2 ,
  \endaligned $$
where in the last step we have used similar derivations of \eqref{eq:useful-identity.1}. Substituting these estimates into the inequality above, we obtain
  $$\E I_3 \lesssim_Q \E \int_s^t \big(\| \omega_3(r) \|_{L^2}^2 + \| v_3(r) \|_{L^2}^2 \big)\, dr \lesssim_{\nu, Q, T} (t-s) \big( \|\omega_3^0 \|_{L^2}^2 + \|v_3^0 \|_{L^2}^2 \big). $$

Summarizing the above estimates on $I_1, I_2$ and $I_3$, we arrive at
  $$\E\big(\|\omega_3(t) -\omega_3(s)\|_{H^{-2}}^2 \big) \le  C_{\nu, Q, T} (t-s) \big(1+ \| \omega_3^0 \|_{L^2}^4+ \| v_3^0 \|_{L^2}^4\big).$$
Combining this estimate with \eqref{lem-a-priori-2.1}, we immediately obtain the desired result.
\end{proof}

In the estimate of $I_1$ in the above proof, we have used the following result.

\begin{lemma}\label{lem-a-priori-fourth-order}
One has, for all $t\in [0,T]$,
  $$\E\big(\|\omega_3(t)\|_{L^2}^4 \big) \le \hat C_{\nu, Q} \big(\|\omega_3^0 \|_{L^2}^4 + \|v_3^0 \|_{L^2}^4\big). $$
\end{lemma}

\begin{proof}
By \eqref{eq:a-priori-omega-3} and It\^o's formula, we have
  $$\aligned
  d \|\omega_3 \|_{L^2}^4 &= 2\|\omega_3 \|_{L^2}^2\, d \|\omega_3 \|_{L^2}^2 + d\big[\|\omega_3 \|_{L^2}^2, \|\omega_3 \|_{L^2}^2 \big]_t \\
  &\le 2\|\omega_3 \|_{L^2}^2 \big(-\nu \|\nabla \omega_3\|_{L^2}^2 + C_{\nu,Q} \|\nabla v_3\|_{L^2}^2\big)\, dt + 4 \sum_k \< \omega_{3},\nabla^{\perp}v_{3} \cdot \nabla \sigma_k^3 \>^2\, dt\\
  &\quad +4 \|\omega_3 \|_{L^2}^2 \sum_k \< \omega_{3},\nabla^{\perp}v_{3} \cdot \nabla \sigma_k^3 \>\, d W^k_t,
  \endaligned $$
where $C_{\nu,Q}= \nu^{-1} \|\nabla Q_{H,3}(0)\|_M^2 + \|\nabla^2 Q_3(0)\|_M$. By Cauchy's inequality,
  $$\sum_k \< \omega_3,\nabla^{\perp}v_{3} \cdot \nabla \sigma_k^3 \>^2 \le \|\omega_3 \|_{L^2}^2 \sum_k \|\nabla^{\perp}v_{3} \cdot \nabla \sigma_k^3 \|_{L^2}^2 $$
where the sum on the right-hand side is the same as \eqref{eq:useful-identity}, and thus by \eqref{eq:useful-identity.1},
  $$\sum_k \< \omega_3,\nabla^{\perp}v_{3} \cdot \nabla \sigma_k^3 \>^2 \le \|\omega_3 \|_{L^2}^2 \|\nabla^2 Q_3(0)\|_M \|\nabla v_{3}\|_{L^2}^2. $$
Substituting this estimate into the inequality above gives us
  \begin{equation}\label{lem-a-priori-fourth-order.1}
  d \|\omega_3 \|_{L^2}^4\le \tilde C_{\nu,Q} \|\omega_3 \|_{L^2}^2 \|\nabla v_3\|_{L^2}^2\, dt +4 \|\omega_3 \|_{L^2}^2 \sum_k \< \omega_{3},\nabla^{\perp}v_{3} \cdot \nabla \sigma_k^3 \>\, d W^k_t.
  \end{equation}
Introduce the stopping time $\tau_n= \inf\{t\ge 0: \|\omega_3(t) \|_{L^2} \ge n\}$; we have
  $$\aligned
  & \E \int_0^{t\wedge \tau_n} \|\omega_3(s) \|_{L^2}^4 \sum_k \< \omega_3(s),\nabla^{\perp}v_3(s) \cdot \nabla \sigma_k^3 \>^2 \, d s \\
  & \le n^6 \|\nabla^2 Q_3(0)\|_M\, \E \int_0^{t\wedge \tau_n} \|\nabla v_3(s) \|_{L^2}^2 \, ds\\
  &\le n^6 \|\nabla^2 Q_3(0)\|_M\, \nu^{-1} \|v_3^0 \|_{L^2}^2,
  \endaligned $$
where the last step follows from \eqref{eq:a-priori-v-3}. This implies the last term in \eqref{lem-a-priori-fourth-order.1} is a locally square integrable martingale, thus we have
  $$\aligned
  \E\big(\|\omega_3(t\wedge \tau_n) \|_{L^2}^4\big) &\le \|\omega_3^0 \|_{L^2}^4 + \tilde C_{\nu,Q}\, \E \int_0^{t\wedge \tau_n} \|\omega_3(s) \|_{L^2}^2 \|\nabla v_3(s)\|_{L^2}^2\, ds \\
  &\le \|\omega_3^0 \|_{L^2}^4 + \tilde C_{\nu,Q}\, \E \bigg[\sup_{s\in [0,T]} \|\omega_3(s) \|_{L^2}^2 \int_0^t \|\nabla v_3(s) \|_{L^2}^2\, ds \bigg] \\
  &\le \|\omega_3^0 \|_{L^2}^4 + \tilde C_{\nu,Q}\, \nu^{-1} \|v_3^0 \|_{L^2}^2\, \E \bigg[\sup_{s\in [0,T]} \|\omega_3(s) \|_{L^2}^2 \bigg],
  \endaligned $$
which, combined with Lemma \ref{lem-a-priori-1}, gives us
  $$\E\big(\|\omega_3(t\wedge \tau_n) \|_{L^2}^4\big) \le \hat C_{\nu,Q} \big(\|\omega_3^0 \|_{L^2}^4 + \|v_3^0 \|_{L^2}^4\big). $$
By Fatou's lemma, letting $n\to \infty$ we finish the proof.
\end{proof}

\subsection{Proof of Theorem \ref{thm:well-posedness}} \label{subsec-proof-well-posedness}

This subsection is devoted to the proof of Theorem \ref{thm:well-posedness}. To show the existence of weak solutions with desired regularity, we may adopt the classical methods of Galerkin approximation and compactness arguments. Thanks to the a priori estimates in Section \ref{sec-a-priori}, the proof is quite standard and can be found in many references, see for instance \cite{FG95, FGL21a} and also \cite[Section 3]{FGL21b}. Therefore, we omit the details of the proof of existence part, and concentrate on the pathwise uniqueness.

\begin{proof}[Proof of Theorem \ref{thm:well-posedness}: pathwise uniqueness]
Let $(\tilde\omega_3, \tilde v_3)$ and $(\bar\omega_3, \bar v_3)$ be two solutions to \eqref{eq:stoch-2D-3C-model-Ito} defined on the same probability space $\Omega$ corresponding to the same initial data and Brownian motions $\{W^k \}_k$, such that $\tilde\omega_3$ and $\bar\omega_3$ (resp. $\tilde v_3$ and $\bar v_3$) fulfill the estimate \eqref{thm:well-posedness.2} (resp.  \eqref{thm:well-posedness.1}). According to the discussions below Definition \ref{sect-5-definition}, we have the decompositions
  $$\omega_3:= \tilde \omega_3- \bar\omega_3= V_{\tilde \omega_3} - V_{\bar\omega_3} + M_{\tilde \omega_3}- M_{\bar\omega_3} , \quad v_3:= \tilde v_3-\bar v_3= V_{\tilde v_3} - V_{\bar v_3} + M_{\tilde v_3}- M_{\bar v_3}, $$
where $V_{\tilde \omega_3}, V_{\bar\omega_3}$ and $V_{\tilde v_3}, V_{\bar v_3}$ take values in $W^{1,2}(0,T; H^{-1})$, while $M_{\tilde \omega_3}, M_{\bar\omega_3}$ and $M_{\tilde v_3}, M_{\bar v_3}$ are $L^2$-valued continuous local martingales. One can check that the conditions of \cite[Theorem 2.13]{RL18} are verified for $\omega_3$ and $v_3$, and thus the It\^o formula therein is applicable.

More precisely, letting ($K$ is the Biot-Savart kernel on $\T^2$)
  $$\aligned
  v_H &= \tilde v_H - \bar v_H= K\ast(\tilde \omega_3 -\bar\omega_3)= K\ast \omega_3,\\
  \omega_H &= \tilde \omega_H -\bar\omega_H = \nabla^\perp(\tilde v_3- \bar v_3) = \nabla^\perp v_3,
  \endaligned $$
then we have
  $$\aligned
  d \omega_3 &= \big(-v_H\cdot\nabla \tilde \omega_3 - \bar v_H\cdot\nabla \omega_3 + \nu \Delta \omega_3 + \mathcal L \omega_3 + \nabla\cdot[\nabla Q_{H,3}(0)\, \omega_H] \big)\, dt\\
  &\quad - \sum_k \big(\sigma_k^H \cdot\nabla \omega_3 - \omega_H \cdot\nabla \sigma_k^3\big) \, dW^k_t ,\\
  d v_3 &= \big(-v_H\cdot\nabla \tilde v_3 - \bar v_H\cdot\nabla v_3 + \nu \Delta v_3 + \mathcal L v_3 \big)\, dt - \sum_k \sigma_k^H \cdot\nabla v_3 \, dW^k_t.
  \endaligned$$
By the It\^o formula in \cite[Theorem 2.13]{RL18},
  \begin{equation*}
  \aligned
  d\| \omega_3\|_{L^2}^2 &= -2\< \omega_3, v_H\cdot\nabla \tilde \omega_3\>\, dt -2\nu \|\nabla \omega_3\|_{L^2}^2\, dt - \<\nabla \omega_3, Q_H(0) \nabla \omega_3\>\, dt\\
  &\quad -2\<\nabla \omega_3,\nabla Q_{H,3}(0)\, \omega_H\>\, dt + 2 \sum_k \< \omega_3, \omega_H\cdot\nabla \sigma_k^3 \>\, d W^k_t \\
  &\quad + \sum_k \big\|\sigma_k^H \cdot\nabla \omega_3 - \omega_H \cdot\nabla \sigma_k^3 \big\|_{L^2}^2\, dt,
  \endaligned
  \end{equation*}
where we have used $\< \omega_3, \bar v_H\cdot\nabla \omega_3\> =0 = \big\< \omega_3, \sigma_k^H \cdot\nabla \omega_3 \big\> $ since $\bar v_H$ and $\sigma_k^H$ are divergence free. It is easy to know that
  $$\aligned
  I&:= \sum_k \big\|\sigma_k^H \cdot\nabla \omega_3 - \omega_H \cdot\nabla \sigma_k^3 \big\|_{L^2}^2 \\
  &= \sum_k \Big(\big\|\sigma_k^H \cdot\nabla \omega_3 \big\|_{L^2}^2 +\big\| \omega_H \cdot\nabla \sigma_k^3 \big\|_{L^2}^2 - 2\big\<\sigma_k^H \cdot\nabla \omega_3, \omega_H \cdot\nabla \sigma_k^3\big> \Big) \\
  &= \<\nabla \omega_3, Q_H(0) \nabla \omega_3\> -\< \omega_H, \nabla^2 Q_3(0)\, \omega_H\> + 2 \<\nabla \omega_3, \nabla Q_{H,3}(0)\, \omega_H\> ,
  \endaligned $$
where in the last step we have used \eqref{eq:matrix-identity-1} and \eqref{eq:matrix-identity-2}; therefore, the identity above reduces to
  \begin{equation}\label{eq:Ito-omega}
  \aligned
  d\| \omega_3\|_{L^2}^2 &= -2\< \omega_3, v_H\cdot\nabla \tilde \omega_3\>\, dt -2\nu \|\nabla \omega_3\|_{L^2}^2\, dt \\
  &\quad + 2 \sum_k \< \omega_3, \omega_H\cdot\nabla \sigma_k^3 \>\, d W^k_t -\< \omega_H, \nabla^2 Q_3(0)\, \omega_H\> \, dt.
  \endaligned
  \end{equation}
In the same way,
  \begin{equation}\label{eq:Ito-v}
  d\| v_3\|_{L^2}^2 = -2\< v_3, v_H\cdot\nabla \tilde v_3\>\, dt -2\nu \|\nabla v_3\|_{L^2}^2\, dt.
  \end{equation}

The estimate of the first term on the right-hand side of \eqref{eq:Ito-omega} is quite standard: recalling that $ v_H=K\ast \omega_3$, we have
  \begin{equation}\label{eq:nonlinearity}
  \aligned
  |\< \omega_3, v_H\cdot\nabla \tilde \omega_3\>| &\le \|\omega_3\|_{L^2} \|v_H\|_{L^\infty} \|\nabla \tilde \omega_3\|_{L^2} \le \|\omega_3\|_{L^2} \|K\ast \omega_3\|_{H^2} \|\nabla \tilde \omega_3\|_{L^2} \\
  &\le \|\omega_3\|_{L^2} \| \omega_3\|_{H^1} \|\nabla \tilde \omega_3\|_{L^2} \le \frac\nu 2 \|\nabla \omega_3\|_{L^2}^2 + \frac1{2\nu} \|\omega_3\|_{L^2}^2 \|\nabla \tilde \omega_3\|_{L^2}^2,
  \endaligned
  \end{equation}
as a result,
  \begin{equation*}
  \aligned
  d\| \omega_3\|_{L^2}^2 + \nu \|\nabla \omega_3\|_{L^2}^2\, dt &\le \nu^{-1} \| \omega_3\|_{L^2}^2 \|\nabla \tilde \omega_3\|_{L^2}^2\, dt -\< \omega_H, \nabla^2 Q_3(0)\, \omega_H\> \, dt \\
  &\quad + 2 \sum_k \< \omega_3, \omega_H\cdot\nabla \sigma_k^3 \>\, d W^k_t.
  \endaligned
  \end{equation*}
Similar calculations yield
  $$ d\| v_3\|_{L^2}^2 + 2\nu \|\nabla v_3\|_{L^2}^2\, dt \le \nu \|\nabla \omega_3\|_{L^2}^2\, dt + \nu^{-1} \| v_3\|_{L^2}^2 \|\nabla \tilde v_3\|_{L^2}^2\, dt .$$

Taking the sum of the two estimates above, we arrive at
  $$\aligned
  d\big(\| \omega_3\|_{L^2}^2 + \| v_3\|_{L^2}^2\big) &\le - \big(2\nu \|\nabla v_3\|_{L^2}^2+\< \omega_H, \nabla^2 Q_3(0)\, \omega_H\> \big) \, dt + 2 \sum_k \< \omega_3, \omega_H\cdot\nabla \sigma_k^3 \>\, d W^k_t \\
  &\quad + \nu^{-1} \big(\|\nabla \tilde \omega_3\|_{L^2}^2 + \|\nabla \tilde v_3\|_{L^2}^2\big) \big(\| \omega_3\|_{L^2}^2 + \|v_3\|_{L^2}^2\big)\, dt .
  \endaligned $$
Recall that we assume $\|\nabla^2 Q_3(0)\|_M \le 2\nu$; hence,
  $$|\< \omega_H, \nabla^2 Q_3(0)\, \omega_H\>| \le \| \omega_H\|_{L^2} \|\nabla^2 Q_3(0)\, \omega_H\|_{L^2} \le 2\nu \| \omega_H\|_{L^2}^2 = 2\nu \|\nabla v_3 \|_{L^2}^2, $$
thus, denoting by $M_t$ the martingale part, we obtain
  \begin{equation}\label{uniqueness}
  d\big(\| \omega_3\|_{L^2}^2 + \| v_3\|_{L^2}^2\big) \le dM_t + \nu^{-1} \big(\|\nabla \tilde \omega_3\|_{L^2}^2 + \|\nabla \tilde v_3\|_{L^2}^2\big) \big(\| \omega_3\|_{L^2}^2 + \|v_3\|_{L^2}^2\big)\, dt.
  \end{equation}
We have, by It\^o's isometry and Cauchy's inequality,
  $$\aligned
  \E [M]_t &= 4\, \E \bigg[\sum_k \int_0^t \big\< \omega_3(s), \omega_H(s) \cdot\nabla \sigma_k^3 \big\>^2 \, ds \bigg]\\
  &\le 4\,\E \bigg[\sum_k \int_0^t \| \omega_3(s)\|_{L^2}^2 \big\| \omega_H(s) \cdot\nabla \sigma_k^3 \big\|_{L^2}^2 \, ds \bigg] \\
  &\le 8\nu\, \E \bigg[ \int_0^t \| \omega_3(s)\|_{L^2}^2 \| \omega_H(s) \|_{L^2}^2 \, ds \bigg],
  \endaligned $$
where we have used \eqref{eq:useful-identity.1} and the bound $\|\nabla^2 Q_3(0) \|_M \le 2\nu$. Note that $ \omega_H = \nabla^\perp v_3 = \nabla^\perp \tilde v_3 -\nabla^\perp \bar v_3$, by \eqref{thm:well-posedness.1}, we have
  $$\aligned
  \E [M]_t &\le 8\nu \, \E \bigg[ \bigg( \sup_{s\in [0,T]} \| \omega_3(s)\|_{L^2}^2 \bigg) \int_0^T \| \nabla^\perp v_3(s) \|_{L^2}^2 \, ds \bigg] \\
  &\le C \| v_3^0 \|_{L^2}^2\, \E \bigg( \sup_{s\in [0,T]} \|\omega_3(s)\|_{L^2}^2 \bigg) \\
  &\le C_{\nu,Q} \big(\| v_3^0 \|_{L^2}^2+ \| \omega_3^0 \|_{L^2}^2\big)^2,
  \endaligned $$
where in the third step we have used \eqref{thm:well-posedness.2}. This implies that $M_t$ is a square integrable martingale. Define a positive and decreasing process
  $$\rho_t := \exp\bigg[-\nu^{-1} \int_0^t \big(\|\nabla \tilde \omega_3(s)\|_{L^2}^2 + \|\nabla \tilde v_3(s) \|_{L^2}^2\big)\, ds\bigg], \quad t\ge 0;$$
then $d\rho_t = -\nu^{-1} \rho_t \big(\|\nabla \tilde \omega_3(t)\|_{L^2}^2 + \|\nabla \tilde v_3(t) \|_{L^2}^2\big)\, dt$. By \eqref{uniqueness} and It\^o's formula,
  $$d \big[\rho_t \big(\| \omega_3\|_{L^2}^2 + \|v_3\|_{L^2}^2\big) \big] \le \rho_t \, dM_t. $$
Integrating from $0$ to $t$ and taking expectation lead to
  $$\E\big[\rho_t \big(\|\omega_3(t) \|_{L^2}^2 + \|v_3(t) \|_{L^2}^2\big) \big] \le \E \int_0^t \rho_s \, dM_s =0, $$
where in the last step we used the fact that $\big\{ \int_0^t \rho_s \, dM_s \big\}_t$ is a martingale since $0< \rho_t \le 1$ and $\{M_t\}_t$ is a square integrable martingale. We conclude that $\|\omega_3(t) \|_{L^2} = \|v_3(t) \|_{L^2} \equiv 0$ and the proof is finished.
\end{proof}

\section{Scaling limit and proof of Theorem \ref{thm:Boussinesq-general-result}}

In this part we first show that, under Hypothesis \ref{basic-hypotheses} in Section \ref{subsec-scaling}, the stochastic 2D-3C models \eqref{eq:Ito-system-ell} with the noise $W^\ell$ converge weakly to deterministic model; then we prove the technical results in Proposition \ref{prop:limit-matrix}.

\subsection{Proof of Theorem \ref{thm:Boussinesq-general-result}} \label{sect-proof-scaling-limit}

We first recall the setting: we are given covariance functions $\{Q^\ell\}_{\ell\in (0,1)}$ satisfying Hypothesis \ref{basic-hypotheses}, and $\{\mathbb Q^\ell \}_{\ell\in (0,1)}$ are the corresponding covariance operators; let $W^\ell$ be a $\mathbb Q^\ell$-Wiener process with the expression
  $$W^\ell(t,x)= \sum_k \sigma^{\ell,H}_k(x) W^k_t + \sum_k \sigma^{\ell,3}_k(x) W^k_t e_3, $$
where $\sigma^{\ell,H}_k= \big(\sigma^{\ell,1}_k, \sigma^{\ell,2}_k \big)$ are divergence free vector fields on $\T^2$, and $\{W^k \}_k$ are independent standard Brownian motions. For any $\ell\in (0,1)$, by Theorem \ref{thm:well-posedness}, the following system
  \begin{equation}\label{eq:Ito-ell}
  \left\{ \aligned
  d\omega^\ell_{3} + v^\ell_H \cdot \nabla \omega^\ell_{3} \, dt &= - \sum_k \big(\sigma^{\ell,H}_k \cdot \nabla \omega^\ell_{3} - \omega^\ell_H\cdot \nabla \sigma^{\ell,3}_k \big)\, dW^k_t \\
  &\quad + \big(\nu \Delta +\mathcal L^\ell\big)\omega^\ell_3 \, dt + \nabla\cdot \big[\nabla Q_{H,3}^{\ell} (0)\, \omega^\ell_H \big] \, dt,\\
  dv^\ell_3+ v^\ell_H \cdot \nabla v^\ell_3 \, dt &= - \sum_k \sigma^{\ell,H}_k \cdot\nabla v^\ell_{3} \, dW^k_t + \big(\nu \Delta +\mathcal L^\ell \big) \omega^\ell_3 \, dt
  \endaligned \right.
  \end{equation}
has a probabilistic weak solution $\big(\omega^\ell_3, v^\ell_3\big)$ satisfying the estimates:
  \begin{align}
  & \P \mbox{-a.s. for all } t\in[0,T], \quad \big\|v^\ell_3(t) \big\|_{L^2}^2 + 2\nu \int_0^t \big\|\nabla v^\ell_3(s) \big\|_{L^2}^2\, ds \le \|v_3(0) \|_{L^2}^2, \label{estimates-v-ell} \\
  & \E\bigg[\sup_{t\in [0,T]} \big\|\omega^\ell_3(t) \big\|_{L^2}^2 + \nu \int_0^T \big\|\nabla \omega^\ell_3(s) \big\|_{L^2}^2\, ds \bigg] \le C_{\nu, Q} \big(\|v_3(0) \|_{L^2}^2 + \|\omega_3(0) \|_{L^2}^2\big), \label{estimates-omega-ell}
  \end{align}
where $C_{\nu, Q}$ is a constant depending on $\nu,\, \sup_{\ell\in (0,1)} \|\nabla Q^\ell_{H,3}(0)\|_M$ and $\sup_{\ell\in (0,1)} \|\nabla^2 Q^\ell_3(0)\|_M$, which are finite by Hypothesis \ref{basic-hypotheses}. We point out that, since we are dealing with weak solutions, the probability and expectation should be written more accurately as $\P^\ell$ and $\E^\ell$, but we omit such dependence for simplicity of notation. Finally, using the equations in \eqref{eq:Ito-ell}, we can also obtain similar estimates as in Lemma \ref{lem-a-priori-2}.

Let $\eta^\ell$ be the law of $\big(\omega^\ell_3, v^\ell_3\big)$, $\ell\in (0,1)$; then the above arguments show that $\{\eta^\ell \}_\ell$ is bounded in probability in $L^2(0,T; H^1) \cap W^{\alpha, 2} (0,T; H^{-2})$, the latter being compactly embedded in $L^2(0,T; L^2)$. Therefore, the family of laws $\{\eta^\ell \}_\ell$ is tight on $L^2(0,T; L^2)$. Then, by the Prohorov theorem, we can find a subsequence (not relabelled for simplicity) of  $\{\eta^\ell \}_\ell$ converging weakly, in the topology of $L^2(0,T; L^2)$, to some probability measure $\eta$. Moreover, by Skorohod's representation theorem, there exists a new probability space $\big(\hat\Omega, \hat{\mathcal F}, \hat\P \big)$, a sequence of processes $\big(\hat \omega^\ell_3, \hat v^\ell_3\big)$ and a limit process $(\hat \omega_3, \hat v_3)$ defined on $\hat\Omega$, such that
\begin{itemize}
\item $(\hat \omega_3, \hat v_3) \stackrel{d}{\sim} \eta$ and $\big(\hat \omega^\ell_3, \hat v^\ell_3\big) \stackrel{d}{\sim} \eta^\ell$ for any $\ell\in (0,1)$;
\item $\hat\P$-a.s., $\big(\hat \omega^\ell_3, \hat v^\ell_3\big)$ converges as $\ell\to 0$ to $(\hat \omega_3, \hat v_3)$ in the topology of $L^2(0,T; L^2)$.
\end{itemize}
We remark that the family $ \big\{ \big(\hat \omega^\ell_3, \hat v^\ell_3\big) \big\}_\ell$ fulfills the same estimates as \eqref{estimates-v-ell} and \eqref{estimates-omega-ell} with respect to $\hat \P$ and $\hat \E$, respectively; thus, up to a further subsequence, $\big(\hat \omega^\ell_3, \hat v^\ell_3\big)$ converges weakly in $L^2\big(\Omega, C([0,T]; L^2)\cap L^2(0,T; H^1)\big)$ to $(\hat \omega_3, \hat v_3)$; in particular, the limit processes $(\hat \omega_3, \hat v_3)$ enjoy the same bounds. Meanwhile, we can also obtain the existence of a family of Brownian motions $\big\{ \big(\hat W^{\ell,k} \big)_{k\ge 1} \big\}_{\ell\in (0,1)}$ on $\hat \Omega$, such that for any $\ell\in (0,1)$, $\big(\hat \omega^\ell_3, \hat v^\ell_3, \big(\hat W^{\ell,k} \big)_{k\ge 1} \big)$ has the same law as $\big(\omega^\ell_3, v^\ell_3, ( W^{k} )_{k\ge 1} \big)$; we omit the details here.

For any $\phi\in C^\infty(\T^2)$, since $\big(\hat \omega^\ell_3, \hat v^\ell_3, \big(\hat W^{\ell,k} \big)_{k\ge 1} \big)$ has the same law as $\big(\omega^\ell_3, v^\ell_3, ( W^{k} )_{k\ge 1} \big)$, we have
  \begin{equation}\label{eq:ell-hat-omega}
  \aligned
  \big\<\hat \omega^\ell_3(t), \phi\big\> &= \<\omega_3(0),\phi\> + \int_0^t \big\<\hat \omega^\ell_3(s), \hat v^\ell_{H}(s) \cdot\nabla \phi \big\>\, ds -\int_0^t \big\<\hat\omega^\ell_{H}(s), (\nabla Q^\ell_{H,3}(0))^\ast \nabla\phi \big\>\, ds  \\
  &\quad - \int_0^t \Big\<\nabla\hat \omega^\ell_3(s), \Big(\nu I_2+ \frac12 Q^\ell_H(0)\Big) \nabla \phi \Big\>\, ds  \\
  &\quad + \sum_k \int_0^t \big[ \big\<\hat \omega^\ell_3(s), \sigma_k^{\ell,H} \cdot \nabla\phi \big\> - \big\<\sigma_k^{\ell,3}, \hat\omega^\ell_{H}(s)\cdot \nabla \phi\big\> \big]\, d\hat W^{\ell,k}_s ,
  \endaligned
  \end{equation}
  \begin{equation}\label{eq:ell-hat-v}
  \aligned
  \big\<\hat v^\ell_3(t), \phi\big\> &= \<v_3(0),\phi\> + \int_0^t \big\<\hat v^\ell_3(s), \hat v^\ell_{H}(s)\cdot\nabla \phi \big\>\, ds + \sum_k \int_0^t \big\<\hat v^\ell_3(s), \sigma_k^{\ell,H} \cdot \nabla\phi \big\>\, d\hat W^{\ell,k}_s\\
  &\quad -\int_0^t \Big\<\nabla \hat v^\ell_3(s), \Big(\nu I_2 + \frac12 Q^\ell_H(0)\Big) \nabla\phi \Big\>\, ds.
  \endaligned
  \end{equation}

It remains to take the limit $\ell\to 0$ in the above equations \eqref{eq:ell-hat-omega} and \eqref{eq:ell-hat-v}; thanks to the discussions in the previous paragraphs, it is standard to show the convergence of all the terms, including the nonlinear ones, except those involving stochastic integrals, which are the key ingredients for proving that the pair of limit processes $(\hat \omega_3, \hat v_3)$ is a weak solution to the deterministic system (cf. \eqref{eq:2D-3C-model-general-limit})
  \begin{equation}\label{eq:2D-3C-limit}
  \left\{ \aligned
  \partial_t \hat\omega_3 + \hat v_H \cdot \nabla \hat\omega_3 &= \big(\nu \Delta+\mathcal L_{\bar Q} \big)\hat\omega_3 + \nabla\cdot (A\, \hat\omega_H),\\
  \partial_t \hat v_3+ \hat v_H \cdot \nabla \hat v_3 &= \big(\nu \Delta+\mathcal L_{\bar Q} \big) \hat v_3,
  \endaligned \right.
  \end{equation}
where $\hat\omega_3= \nabla^\perp\cdot \hat v_H,\, \hat\omega_H= \nabla^\perp \hat v_3$. Therefore, we concentrate on the martingale parts in \eqref{eq:ell-hat-omega} and \eqref{eq:ell-hat-v}, trying to show that they vanish in a certain sense.

\begin{lemma}\label{lem-marting-zero-limit}
Assume Hypothesis \ref{basic-hypotheses}-(b); then as $\ell\to 0$, the martingale part
  $$ M^\ell_1(t):= \sum_k \int_0^t \big\<\hat v^\ell_3(s), \sigma_k^{\ell,H} \cdot \nabla\phi \big\>\, d\hat W^{\ell,k}_s $$
tends to 0 in the mean square sense.
\end{lemma}

\begin{proof}
By It\^o's isometry,
  $$\aligned
  \E\big( M^\ell_1(t)^2 \big) &= \E\bigg(\sum_k \int_0^t \big\<\hat v^\ell_{3}(s), \sigma^{\ell,H}_k \cdot\nabla \phi \big\>^2 \, ds \bigg) \\
  &= \E\bigg(\sum_k \int_0^t\!\! \int_{(\T^2)^2} \big[\big(\hat v^\ell_{3}(s)\nabla \phi\big)(x) \big]^\ast \sigma^{\ell,H}_k(x) \big[\sigma^{\ell,H}_k(y) \big]^\ast \big(\hat v^\ell_{3}(s)\nabla \phi\big)(y) \,dx dy ds \bigg) \\
  &= \E\bigg( \int_0^t\!\! \int_{(\T^2)^2} \big[\big(\hat v^\ell_{3}(s)\nabla \phi\big)(x) \big]^\ast Q_H^\ell(x-y) \big(\hat v^\ell_{3}(s)\nabla \phi\big)(y) \,dx dy ds \bigg) ;
  \endaligned $$
recalling the definition of the covariance operator $\mathbb Q_H^\ell$, we have
  $$\aligned
  \E\big(M^\ell_1(t)^2 \big)&= \E\bigg( \int_0^t \big\<\hat v^\ell_{3}(s)\nabla \phi, \mathbb Q_H^\ell\big(\hat v^\ell_{3}(s)\nabla \phi \big) \big\>\, ds \bigg) \\
  &\le \E\bigg( \int_0^t \big\|\hat v^\ell_{3}(s)\nabla \phi \big\|_{L^2}^2 \big\| \mathbb Q_H^\ell\big\|_{L^2\to L^2}\, ds \bigg) \\
  &\le \big\| \mathbb Q_H^\ell\big\|_{L^2\to L^2} \|\nabla \phi\|_{L^\infty}^2 T \|v_3(0) \|_{L^2}^2,
  \endaligned $$
where in the last step we have used estimate \eqref{estimates-v-ell}. Thanks to condition (b) in Hypothesis \ref{basic-hypotheses}, we conclude that the last quantity vanishes as $\ell\to 0$.
\end{proof}

Using the uniform bound \eqref{estimates-omega-ell}, we can show in the same way that
  $$ M^\ell_2(t):= \sum_k \int_0^t \big\<\hat \omega^\ell_3(s), \sigma_k^{\ell,H} \cdot \nabla\phi \big\>\, d\hat W^{\ell,k}_s $$
vanishes in the sense of mean square; indeed, the last step of Lemma \ref{lem-marting-zero-limit} becomes
  $$\aligned
  \E\big(M^\ell_2(t)^2 \big) &\le \big\| \mathbb Q_H^\ell\big\|_{L^2\to L^2} \|\nabla \phi\|_{L^\infty}^2\, \E \int_0^t \big\|\hat \omega^\ell_{3}(s) \big\|_{L^2}^2\, ds \\
  & \le \big\| \mathbb Q_H^\ell\big\|_{L^2\to L^2} \|\nabla \phi\|_{L^\infty}^2 T C_{\nu, Q} \big(\|v_3(0) \|_{L^2}^2 + \|\omega_3(0) \|_{L^2}^2\big)
  \endaligned$$
which tends to 0 as $\ell\to 0$. Finally, we consider the other martingale in \eqref{eq:ell-hat-omega}.

\begin{lemma}
Assume condition (b) in Hypothesis \ref{basic-hypotheses}. The martingale part
  $$ M^\ell_3(t):= \sum_k \int_0^t \big\<\sigma_k^{\ell,3}, \hat\omega^\ell_{H}(s)\cdot \nabla \phi\big\> \, d\hat W^{\ell,k}_s $$
vanishes as $\ell\to 0$.
\end{lemma}

\begin{proof}
Again by It\^o's isometry, we have
  $$\aligned
  \E\big(M^\ell_3(t)^2 \big) &= \E\int_0^t \sum_k\big\<\sigma^{\ell,3}_k, \hat\omega^\ell_H(s) \cdot \nabla  \phi \big\>^2 \, ds.
  \endaligned $$
It holds that
  $$\aligned
  \sum_k\big\<\sigma^{\ell,3}_k, \hat\omega^\ell_H(s) \cdot \nabla  \phi \big\>^2
  &= \int_{(\T^2)^2} Q^\ell_3(x-y) (\hat\omega^\ell_H(s) \cdot \nabla  \phi)(x) (\hat\omega^\ell_H(s) \cdot \nabla  \phi)(y)\, dx dy \\
  &= \big\<\hat\omega^\ell_H(s) \cdot \nabla  \phi, \mathbb Q^\ell_3 \big(\hat\omega^\ell_H(s) \cdot \nabla  \phi \big)\big> \\
  &\le \big\| \mathbb Q^\ell_3 \big\|_{L^2\to L^2} \big\| \hat\omega^\ell_H(s) \cdot \nabla  \phi\big\|_{L^2}^2 \\
  &\le \big\| \mathbb Q^\ell_3 \big\|_{L^2\to L^2} \|\nabla  \phi \|_{L^\infty}^2 \| \nabla \hat v_3^\ell(s) \|_{L^2}^2 ,
  \endaligned $$
where we have used $\hat\omega^\ell_H = \nabla^\perp \hat v_3^\ell$.  Thanks to the estimate \eqref{eq:a-priori-v-3}, we arrive at
  $$\E \big(M^\ell_3(t)^2 \big) \le \big\| \mathbb Q^\ell_3 \big\|_{L^2\to L^2} \|\nabla  \phi \|_{L^\infty}^2 \, \E \int_0^t \| \nabla \hat v_3^\ell(s) \|_{L^2}^2 \, ds \lesssim_{\nu} \big\| \mathbb Q^\ell_3 \big\|_{L^2\to L^2} \|\nabla  \phi \|_{L^\infty}^2 \|v_3(0) \|_{L^2}^2 ,$$
the last quantity vanishes due to condition (b) in Hypothesis \ref{basic-hypotheses}.
\end{proof}

Finally we present

\begin{proof}[Proof of Theorem \ref{thm:Boussinesq-general-result}]
Summarizing the arguments starting from Lemma \ref{lem-marting-zero-limit}, we see that the martingale parts in \eqref{eq:ell-hat-omega} and \eqref{eq:ell-hat-v} vanish in the sense of mean square. Moreover, the discussions above Lemma \ref{lem-marting-zero-limit} give us the convergence of other terms in the equations \eqref{eq:ell-hat-omega} and \eqref{eq:ell-hat-v}. Therefore, we have proved the first assertion in Theorem \ref{thm:Boussinesq-general-result}.

We turn to showing the uniqueness of weak solutions to the system \eqref{eq:2D-3C-model-general-limit} in the class $L^\infty(0,T; L^2) \cap L^2(0,T; H^1)$; the proof is similar to the uniqueness part of Theorem \ref{thm:well-posedness}, using now the Lions-Magenes lemma. Let $(\tilde \omega_3, \tilde v_3)$ and $(\bar\omega_3, \bar v_3 )$ be two weak solutions to \eqref{eq:2D-3C-model-general-limit} in $L^\infty(0,T; L^2) \cap L^2(0,T; H^1)$, with the same initial data, define $\omega_3= \tilde \omega_3 - \bar\omega_3$, $v_3= \tilde v_3 - \bar v_3$; then we have
  $$\left\{ \aligned
  \partial_t \omega_3 + v_H \cdot\nabla \tilde \omega_3 + \bar v_H \cdot\nabla \omega_3 &= \big(\nu \Delta+\mathcal L_{\bar Q} \big) \omega_3 + \nabla\cdot (A\, \omega_H),\\
  \partial_t v_3+ v_H \cdot\nabla \tilde v_3 + \bar v_H \cdot\nabla v_3 &= \big(\nu \Delta+\mathcal L_{\bar Q} \big) v_3,
  \endaligned \right. $$
where $v_H= \tilde v_H- \bar v_H$ and $\omega_H= \tilde \omega_H- \bar \omega_H$. Similarly to the discussions below Definition \ref{sect-5-definition}, we have $\omega_3 \in L^2(0,T; H^1)$ and $\partial_t \omega_3 \in L^2(0,T; H^{-1})$; thus by the Lions-Magenes lemma,
  $$\aligned
  \frac{d}{dt} \|\omega_3 \|_{L^2}^2 &= -2\<\omega_3, v_H\cdot\nabla \tilde \omega_3\> -2\nu \|\nabla \omega_3 \|_{L^2}^2 - \<\nabla \omega_3,\bar Q \nabla \omega_3\>- 2\<\nabla \omega_3,A \,\omega_H \> \\
  &\le -\nu \|\nabla \omega_3 \|_{L^2}^2 + \nu^{-1} \| \omega_3 \|_{L^2}^2 \|\nabla \tilde \omega_3 \|_{L^2}^2 - \<\nabla \omega_3,\bar Q \nabla\omega_3\>- 2\<\nabla \omega_3,A\, \omega_H \>,
  \endaligned $$
where in the second step we have used the estimate \eqref{eq:nonlinearity}. Similarly,
  $$\frac{d}{dt} \| v_3 \|_{L^2}^2 \le \nu \|\nabla \omega_3 \|_{L^2}^2 + \nu^{-1} \| v_3 \|_{L^2}^2 \|\nabla \tilde v_3 \|_{L^2}^2 -2\nu \|\nabla v_3 \|_{L^2}^2 - \<\nabla v_3,\bar Q \nabla v_3\> . $$
Noting that $\omega_H=\nabla^\perp v_3$; summing up the above two inequalities leads to
  $$\aligned
  \frac{d}{dt} \big(\|\omega_3 \|_{L^2}^2+ \|v_3 \|_{L^2}^2\big) & \le \nu^{-1} \big( \|\nabla \tilde \omega_3 \|_{L^2}^2 + \|\nabla \tilde v_3 \|_{L^2}^2 \big) \big(\|\omega_3 \|_{L^2}^2 + \|v_3 \|_{L^2}^2 \big) \\
  &\quad - \<\nabla \omega_3,\bar Q \nabla \omega_3\> - \<\nabla v_3,\bar Q \nabla v_3\>- 2\<\nabla \omega_3,A \nabla^\perp v_3 \> \\
  &\le \nu^{-1} \big( \|\nabla \tilde \omega_3 \|_{L^2}^2 + \|\nabla \tilde v_3 \|_{L^2}^2 \big) \big(\|\omega_3 \|_{L^2}^2 + \|v_3 \|_{L^2}^2 \big)
  \endaligned $$
where the second step follows from condition \eqref{eq:cond-unique-limit}. Since $\tilde \omega_3$ and $\tilde v_3$ belong to $L^2(0,T;H^1)$, we finish the proof by applying Gronwall's inequality.
\end{proof}

\subsection{Proof of Proposition \ref{prop:limit-matrix}} \label{subsec-proof-prop}

We first make some preparations. Recall that $G$ is the Green function on $\T^2$; the Green function on $\R^2$ admits the expression
  $$ G_{\R^2}(x) = -\frac1{2\pi} \log |x|; $$
there exists a smooth function $\zeta:\T^2 \to \R$ such that
  \begin{equation} \label{eq:Green-funct}
  G(x)= G_{\R^2}(x) + \zeta(x), \quad x\in \T^2.
  \end{equation}
Given a function $f\in C(\T^2)$ with compact support in $\big(-\frac12, \frac12\big)^2$, we will regard it also as a function on $\R^2$, still supported in $\big(-\frac12, \frac12\big)^2$; for any $\ell\in (0,1)$, we denote $f_\ell(x) = \ell^{-2} f(\ell^{-1} x),\, x\in \T^2$. Let $\T^2_{\ell^{-1}} = \big[-\frac12 \ell^{-1}, \frac12 \ell^{-1} \big]^2$ be the torus of size $\ell^{-1} $. We first prove the following formula.

\begin{lemma}\label{lem-expression-derivative-G}
Let $\phi, \psi \in C_c\big( (-\frac12, \frac12)^2 \big)$. Then we have
  $$\aligned
  \<\partial_i G\ast \phi_\ell, \partial_j G\ast \psi_\ell\> &= \int_{\T^2_{\ell^{-1}}}\! (\partial_i G_{\R^2} \ast \phi)(x) \, (\partial_j G_{\R^2} \ast \psi)(x) \, dx \\
  &\quad + \ell  \int_{\T^2_{\ell^{-1}}}\! (\partial_i G_{\R^2} \ast \phi)(x) \bigg[\int_{\T^2} \partial_j \zeta(\ell(x-z)) \, \psi(z)\, dz \bigg]\, dx \\
  &\quad + \ell  \int_{\T^2_{\ell^{-1}}}\! \bigg[\int_{\T^2} \partial_i \zeta(\ell(x-z)) \, \phi(z)\, dz \bigg] (\partial_j G_{\R^2} \ast \psi)(x) \, dx \\
  &\quad + \ell^2 \int_{\T^2_{\ell^{-1}}}\! \bigg[\int_{\T^2} \partial_i \zeta(\ell(x-z)) \, \phi(z)\, dz \bigg] \bigg[\int_{\T^2} \partial_j \zeta(\ell(x-z)) \, \psi(z)\, dz \bigg]\, dx.
  \endaligned $$
\end{lemma}

\begin{proof}
We have
  $$(\partial_i G\ast \phi_\ell)(x) =  \int_{\T^2} \partial_i G(x-y)\, \phi_\ell(y)\, dy = \int_{\T^2_{\ell^{-1}}} \partial_i G(x- \ell y')\, \phi(y')\, dy', $$
where in the second step we have changed the variable $y =\ell y'$. As $ \phi$ has compact support in $(-\frac12, \frac12)^2$, we arrive at
  $$(\partial_i G\ast \phi_\ell)(x)= \int_{\T^2} \partial_i G(x- \ell y) \, \phi(y)\, dy; $$
in the same way,
  $$(\partial_j G\ast \psi_\ell)(x)= \int_{\T^2} \partial_j G(x- \ell z) \, \psi(z)\, dz. $$
Therefore,
  $$\aligned
  \<\partial_i G\ast \phi_\ell, \partial_j G\ast \psi_\ell\>&= \int_{\T^2} \bigg[\int_{\T^2} \partial_i G(x- \ell y) \, \phi(y)\, dy \bigg] \bigg[\int_{\T^2} \partial_j G(x- \ell z) \, \psi(z)\, dz \bigg] dx\\
  &= \ell^2\!\! \int_{\T^2_{\ell^{-1}}} \bigg[\int_{\T^2} \partial_i G(\ell(x'-y)) \, \phi(y)\, dy \bigg] \bigg[\int_{\T^2} \partial_j G(\ell(x'-z)) \, \psi(z)\, dz \bigg] dx' ,
  \endaligned $$
where we have changed variable $x=\ell x'$. Note that $\partial_i G_{\R^2}(\ell x) = \ell^{-1} \partial_i G_{\R^2}(x)$ for any $x\in \R^2 \setminus \{0\}$, therefore,
  $$\int_{\T^2} \partial_i G_{\R^2} (\ell(x'-y)) \, \phi(y)\, dy =\ell^{-1} (\partial_i G_{\R^2} \ast \phi)(x'). $$
Combining this fact with \eqref{eq:Green-funct}, we obtain the desired expression.
\end{proof}

The next simple result shows that the convolution of $\nabla G_{\R^2}$ and a probability density, with compact support, is close to $\nabla G_{\R^2}$ in the region far from the origin. Let $B_R \subset \R^2$ be the open ball centered at the origin with radius $R>0$.

\begin{lemma}\label{lem-nabla-G}
Let $\psi\in C_c(B_R )$ be a probability density function. Then there exists a constant $C_R>0$ such that
  $$\big|(\nabla G_{\R^2} \ast \psi)(x) - \nabla G_{\R^2}(x)\big| \le \frac{C_R}{|x|^2}, \quad \mbox{for all } |x|\ge (2R) \vee 1. $$
\end{lemma}

\begin{proof}
Note that
  $$\aligned
  (\nabla G_{\R^2} \ast \psi)(x)- \nabla G_{\R^2}(x) &= -\frac1{2\pi} \int_{B_R} \frac{x-y}{|x-y|^2} \psi(y)\, dy + \frac1{2\pi} \frac{x}{|x|^2} \\
  &= -\frac1{2\pi} \int_{B_R} \bigg(\frac{x-y}{|x-y|^2} - \frac{x}{|x|^2} \bigg) \psi(y)\, dy,
  \endaligned $$
therefore,
  $$\aligned
  \big| (\nabla G_{\R^2} \ast \psi)(x)- \nabla G_{\R^2}(x) \big| &\le \frac1{2\pi} \int_{B_R} \bigg| \frac{x-y}{|x-y|^2} - \frac{x}{|x|^2} \bigg| \psi(y)\, dy \\
  &= \frac1{2\pi} \int_{B_R} \frac{\big| |x|^2(x-y) - |x-y|^2 x\big|}{|x-y|^2|x|^2}\, \psi(y)\, dy .
  \endaligned $$
Since $|x|^2(x-y) - |x-y|^2 x = -|x|^2 y -|y|^2 x + 2(x\cdot y) x$, we have
  $$\aligned
  \big| (\nabla G_{\R^2} \ast \psi)(x)- \nabla G_{\R^2}(x) \big| &\le \frac1{2\pi} \int_{B_R} \frac{|x|^2 |y| +|y|^2 |x| + 2|x\cdot y|\, |x|}{|x-y|^2|x|^2}\, \psi(y)\, dy \\
  &\le \frac{1}{2\pi} \int_{B_R} \frac{|x|^2 R +R^2 |x| + 2R |x|^2}{|x-y|^2|x|^2}\, \psi(y)\, dy \\
  &\le \frac{R^2 + 3R}{2\pi} \int_{B_R} \frac{1}{|x-y|^2 }\, \psi(y)\, dy
  \endaligned$$
for all $|x|\ge 1$. Now if $|x|\ge (2R) \vee 1$, we have
  $$\aligned
  \big| (\nabla G_{\R^2} \ast \psi)(x)- \nabla G_{\R^2}(x) \big| &\le \frac{R^2 + 3R}{2\pi|x|^2} \int_{B_R} \frac{|x|^2}{|x-y|^2 }\, \psi(y)\, dy \\
  &\le \frac{2(R^2 + 3R)}{\pi|x|^2}
  \endaligned  $$
since $\frac{|x|^2}{|x-y|^2 } \le 4$ for $|y|\le R$ and $|x|\ge 2R$, and $\psi$ is a probability density function.
\end{proof}

The next lemma gives the limit behavior of integrals of $\partial_i G$ outside a fixed ball.

\begin{lemma} \label{lem-limit-integral-G}
For any $R>0$, $i=1,2$, we have
  $$\lim_{\ell\to 0} \frac1{\log \ell^{-1}} \int_{\T^2_{\ell^{-1}} \setminus B_R} (\partial_i G_{\R^2}(x) )^2\, dx = \frac1{4\pi}. $$
\end{lemma}

\begin{proof}
Since $\partial_i G_{\R^2}(x) = -\frac1{2\pi} \frac{x_i}{|x|^2}$, we have
  $$\aligned
  \int_{\T^2_{\ell^{-1}} \setminus B_R} (\partial_1 G_{\R^2}(x) )^2\, dx &= \frac1{4\pi^2} \int_{\T^2_{\ell^{-1}} \setminus B_R} \frac{x_1^2}{|x|^4}\, dx,
  \endaligned $$
by changing of variable $(x_1,x_2) \to (x_2, x_1)$, it is equal to
  $$\aligned
  \frac1{4\pi^2} \int_{\T^2_{\ell^{-1}} \setminus B_R} \frac{x_2^2}{|x|^4}\, dx &= \int_{\T^2_{\ell^{-1}} \setminus B_R} (\partial_2 G_{\R^2}(x) )^2\, dx.
  \endaligned $$
Therefore,
  \begin{equation} \label{lem-limit-integral-G.1}
  \int_{\T^2_{\ell^{-1}} \setminus B_R} (\partial_1 G_{\R^2}(x) )^2\, dx= \frac{1}{8\pi^2} \int_{\T^2_{\ell^{-1}} \setminus B_R} \frac{dx}{|x|^2} .
  \end{equation}
We have, for any $\ell< R^{-1}$,
  $$\int_{R\le |x| \le (2\ell)^{-1}}  \frac{dx}{|x|^2} \le \int_{\T^2_{\ell^{-1}} \setminus B_R} \frac{dx}{|x|^2} \le \int_{R\le |x| \le \ell^{-1}}  \frac{dx}{|x|^2}, $$
thus,
  $$2\pi \log \frac{\ell^{-1}}{2R} \le \int_{\T^2_{\ell^{-1}} \setminus B_R} \frac{dx}{|x|^2} \le 2\pi \log \frac{\ell^{-1}}{R} .$$
This implies
  $$\lim_{\ell\to 0} \frac{1}{\log \ell^{-1}} \int_{\T^2_{\ell^{-1}} \setminus B_R} \frac{dx}{|x|^2} = 2\pi. $$
Combining this limit with \eqref{lem-limit-integral-G.1}, we finish the proof.
\end{proof}

We are ready to prove the following key result.

\begin{proposition}\label{prop-growth-norm}
Let $\phi, \psi \in C_c\big( (-\frac12, \frac12)^2 \big)$ be radially symmetric probability density functions. Then for $i,j\in \{1,2\}$,
  $$\lim_{\ell\to 0} \frac{\<\partial_i G\ast \phi_\ell, \partial_j G\ast \psi_\ell\>}{\log \ell^{-1}} = \frac1{4\pi} \delta_{i,j}, $$
where $\delta_{i,j}$ is Kronecker's delta, i.e. it equals 1 if $i=j$ and $0$ if $i\ne j$.
\end{proposition}

\begin{proof}
Recall Lemma \ref{lem-expression-derivative-G} for the expression of $\<\partial_i G\ast \phi_\ell, \partial_j G\ast \psi_\ell\>$; we denote the four terms by $I^\ell_n,\, n=1,2,3,4$.

\emph{Step 1}. We first show that
  \begin{equation}\label{prop-growth-norm.1}
  \lim_{\ell\to 0} \frac{|I^\ell_n|}{\log \ell^{-1}} =0, \quad n=2,3,4.
  \end{equation}
Indeed, for $I^\ell_2$, since $\zeta\in C^\infty(\T^2)$, we have
  $$\bigg|\int_{\T^2} \partial_j \zeta(\ell(x-z)) \, \psi(z)\, dz \bigg| \le \|\partial_j \zeta \|_{L^\infty} \int_{\T^2} \psi(z)\, dz = \|\partial_j \zeta \|_{L^\infty}, $$
therefore,
  $$\aligned
  |I^\ell_2| &\le \ell \|\partial_j \zeta \|_{L^\infty} \int_{\T^2_{\ell^{-1}}}\! \big| (\partial_i G_{\R^2} \ast \phi)(x) \big| \, dx \\
  &= \ell \|\partial_j \zeta \|_{L^\infty} \bigg(\int_{B_R} + \int_{\T^2_{\ell^{-1}} \setminus B_R} \bigg) \big| (\partial_i G_{\R^2} \ast \phi)(x) \big| \, dx
  \endaligned $$
for some fixed $R\in [1, (2\ell)^{-1})$. The first integral is bounded by some constant $C_R$; by Lemma \ref{lem-nabla-G}, one has  $\big| (\partial_i G_{\R^2} \ast \phi)(x) \big| \le C'_R/|x|$ for all $|x|\ge R$, and thus
  $$\int_{\T^2_{\ell^{-1}} \setminus B_R} \big| (\partial_i G_{\R^2} \ast \phi)(x) \big| \, dx \le \int_{R\le |x|\le \ell^{-1}} \frac{C'_R}{|x|}\, dx = 2\pi C'_R (\ell^{-1} -R). $$
Summarizing these estimates we arrive at $|I^\ell_2| \le C_{1,R} \ell + C_{2,R}$, which implies that \eqref{prop-growth-norm.1} holds for $n=2$. The proofs for the other two limits are similar.

\emph{Step 2}. Now we consider the term $I^\ell_1$, and denote it with more precise notations
  $$J^\ell_{i,j}= \int_{\T^2_{\ell^{-1}}}\! (\partial_i G_{\R^2} \ast \phi)(x) \, (\partial_j G_{\R^2} \ast \psi)(x) \, dx,\quad 1\le i,j \le 2.$$
We begin with showing that the off-diagonal terms
  $$\aligned
  J^\ell_{1,2} = J^\ell_{2,1} = \int_{\T^2_{\ell^{-1}}} (\partial_1 G_{\R^2} \ast \phi) (x) (\partial_2 G_{\R^2} \ast \psi) (x)\, dx
  \endaligned  $$
vanish for any $\ell\in (0,1)$. Indeed, by symmetry of $\phi$ and $\psi$, one can show that
  $$\aligned
  (\partial_1 G_{\R^2} \ast \phi) (-x_1, x_2)= -(\partial_1 G_{\R^2} \ast \phi) (x_1, x_2),  &\quad (\partial_1 G_{\R^2} \ast \phi) (x_1, -x_2) = (\partial_1 G_{\R^2} \ast \phi) (x_1, x_2), \\
  (\partial_2 G_{\R^2} \ast \psi) (-x_1, x_2)= (\partial_2 G_{\R^2} \ast \psi) (x_1, x_2),  &\quad (\partial_2 G_{\R^2} \ast \psi) (x_1, -x_2) = -(\partial_2 G_{\R^2} \ast \psi) (x_1, x_2).
  \endaligned $$
Due to cancellation of integrals in the four quadrants, we easily conclude that
  \begin{equation} \label{eq:vanishment-cross-terms}
  J^\ell_{1,2} = J^\ell_{2,1}=0.
  \end{equation}

Next, we show that $J^\ell_{1,1}$ has nontrivial limit; similar proof works for $J^\ell_{2,2}$. We have
  $$\aligned
  J^\ell_{1,1} &= \bigg(\int_{B_R}+ \int_{\T^2_{\ell^{-1}} \setminus B_R} \bigg) (\partial_1 G_{\R^2} \ast \phi) (x) (\partial_1 G_{\R^2} \ast \psi) (x)\, dx,
  \endaligned $$
where $B_R$ is the ball centered at the origin with radius $R\ge 1$. It is clear that
  $$\lim_{\ell\to 0} \frac1{\log \ell^{-1}} \int_{B_R} (\partial_1 G_{\R^2} \ast \phi) (x) (\partial_1 G_{\R^2} \ast \psi) (x)\, dx =0 ,$$
thus it holds that
  \begin{equation}\label{eq:limit-J-ell-1-1}
  \lim_{\ell\to 0} \frac{J^\ell_{1,1}}{\log \ell^{-1}} = \lim_{\ell\to 0} \frac1{\log \ell^{-1}} \int_{\T^2_{\ell^{-1}} \setminus B_R} (\partial_1 G_{\R^2} \ast \phi) (x) (\partial_1 G_{\R^2} \ast \psi) (x)\, dx.
  \end{equation}
We have
  $$\aligned
  & \big| (\partial_1 G_{\R^2} \ast \phi) (x) (\partial_1 G_{\R^2} \ast \psi) (x) -(\partial_1 G_{\R^2}(x) )^2 \big| \\
  &\quad \le |(\partial_1 G_{\R^2} \ast \psi) (x)|\, | (\partial_1 G_{\R^2} \ast \phi) (x) - \partial_1 G_{\R^2}(x) | \\
  &\quad \quad + | \partial_1 G_{\R^2} (x)|\, | (\partial_1 G_{\R^2} \ast \psi) (x) - \partial_1 G_{\R^2}(x) |;
  \endaligned $$
thanks to Lemma \ref{lem-nabla-G}, for any $|x|\ge R\ge 1$,
  $$\big| (\partial_1 G_{\R^2} \ast \phi) (x) (\partial_1 G_{\R^2} \ast \psi) (x) -(\partial_1 G_{\R^2}(x) )^2 \big| \le \frac{C_R}{|x|^3}. $$
Since
  $$\limsup_{\ell\to 0} \frac1{\log \ell^{-1}} \int_{\T^2_{\ell^{-1}} \setminus B_R} \frac{C_R}{|x|^3}\, dx \le \limsup_{\ell\to 0} \frac1{\log \ell^{-1}} \int_{\R^2 \setminus B_R} \frac{C_R}{|x|^3}\, dx =0, $$
we conclude from \eqref{eq:limit-J-ell-1-1} that
  $$\lim_{\ell\to 0} \frac{J^\ell_{1,1}}{\log \ell^{-1}} =  \lim_{\ell\to 0} \frac{1}{\log \ell^{-1}} \int_{\T^2_{\ell^{-1}} \setminus B_R} (\partial_1 G_{\R^2}(x) )^2 \, dx = \frac1{4\pi}, $$
where the last identity follows from Lemma \ref{lem-limit-integral-G}.
\end{proof}

Recall the definition of $\Gamma_\ell$ in \eqref{eq:Gamma-1-ell}:
  $$ \Gamma_\ell= 2\sqrt{\kappa}\, \|K\ast \theta_\ell \|_{L^2}^{-1}. $$
Then we have
  $$4\kappa = \Gamma_\ell^2 \|K\ast \theta_\ell \|_{L^2}^2 = 2 \Gamma_\ell^2 \|\partial_1 G\ast \theta_\ell \|_{L^2}^2, $$
by Proposition \ref{prop-growth-norm}, we deduce that
  \begin{equation} \label{Gamma-ell-asymptotics}
  \lim_{\ell \to 0} \Gamma_\ell^2 \log \ell^{-1} = 8\pi \kappa.
  \end{equation}

Now we are ready to provide

\begin{proof}[Proof of Proposition \ref{prop:limit-matrix}]
First we derive the expression for $\nabla Q_{H,3}^\ell(0)$. By the definition of $Q_{H,3}^\ell$, we have
  $$\aligned
  Q_{H,3}^\ell(a)&= \int_{\T^2} \sigma_H^\ell (x)\, \sigma_3^\ell (a-x) \, dx\\
  &= \Gamma_\ell \gamma_\ell \int_{\T^2} (K\ast \theta_\ell)(x)\, (G\ast \chi_\ell) (a-x) \, dx.
  \endaligned $$
Then,
  $$\nabla Q_{H,3}^\ell(a) = \Gamma_\ell\gamma_\ell \int_{\T^2} (K\ast \theta_\ell)(x)\otimes (\nabla G\ast \chi_\ell) (a-x) \, dx, $$
and thus,
  $$\aligned
  \nabla Q_{H,3}^\ell(0) &= -\Gamma_\ell\gamma_\ell \int_{\T^2} (K\ast \theta_\ell)(x)\otimes (\nabla G\ast \chi_\ell) (x) \, dx \\
  &= -\Gamma_\ell\gamma_\ell \begin{pmatrix}
  \<\partial_2 G \ast \theta_\ell, \partial_1 G \ast \chi_\ell\> &\, \<\partial_2 G \ast \theta_\ell, \partial_2 G \ast \chi_\ell\> \smallskip \\
  -\<\partial_1 G \ast \theta_\ell, \partial_1 G \ast \chi_\ell\> &\, -\<\partial_1 G \ast \theta_\ell, \partial_2 G \ast \chi_\ell\>
  \end{pmatrix} \\
  &= -\Gamma_\ell\gamma_\ell \begin{pmatrix}
  0 &\, \<\partial_2 G \ast \theta_\ell, \partial_2 G \ast \chi_\ell\> \smallskip \\
  -\<\partial_1 G \ast \theta_\ell, \partial_1 G \ast \chi_\ell\> &\, 0
  \end{pmatrix} ,
  \endaligned $$
where the last step is due to similar proofs of \eqref{eq:vanishment-cross-terms}.

Combining this formula with Proposition \ref{prop-growth-norm} and the limit \eqref{Gamma-ell-asymptotics}, we immediately deduce that if $\gamma_\ell= o(\Gamma_\ell)$, then all the entries of the matrix $\nabla Q_{H,3}^\ell(0)$ have zero limit. Thus we obtain the first assertion of Proposition \ref{prop:limit-matrix}.

We turn to proving assertion (2). If $\gamma_\ell/\Gamma_\ell \to q_0$ as $\ell\to 0$, using again Proposition \ref{prop-growth-norm} and \eqref{Gamma-ell-asymptotics}, we arrive at
  $$\lim_{\ell\to 0} \Gamma_\ell\gamma_\ell \<\partial_2 G \ast \theta_\ell, \partial_2 G \ast \chi_\ell\> = \lim_{\ell\to 0} \frac{\gamma_\ell}{\Gamma_\ell}\, (\Gamma_\ell^2 \log\ell^{-1})\, \frac{\<\partial_2 G \ast \theta_\ell, \partial_2 G \ast \chi_\ell\>}{\log\ell^{-1}}= 2\kappa q_0 $$
which yields the desired result.

Finally we prove assertion (3) of Proposition \ref{prop:limit-matrix}. We have, for $i,j\in \{1,2\}$,
  $$\partial_i \partial_j Q_3^\ell(0) = -\gamma_\ell^2 \int_{\T^2} (\partial_i G \ast \chi_\ell)(x)\, (\partial_j G\ast \chi_\ell)(x) \, dx= -\gamma_\ell^2\, \<\partial_i G \ast \chi_\ell, \partial_j G\ast \chi_\ell\>. $$
Noting that $\gamma_\ell = O(\Gamma_\ell)$ as $\ell\to 0$, it is clear from the above computations that all the terms are uniformly bounded in $\ell\in (0,1)$.
\end{proof}

\section*{Acknowledgements}

The research of the first author is funded by the European Union (ERC, NoisyFluid, No. 101053472). The second author is grateful to the National Key R\&D Program of China (No. 2020YFA0712700), the National Natural Science Foundation of China (Nos. 11931004, 12090010, 12090014) and the Youth Innovation Promotion Association, CAS (Y2021002).\\
Views and opinions expressed are however those of the authors only and do not necessarily reflect those of the European Union or the European Research Council. Neither the European Union nor the granting authority can be held responsible for them.


\begin{thebibliography}{99} \setlength{\itemsep}{-1pt}

\bibitem{Babin} A. Babin, A. Mahalov, B. Nikolaenko. Regularity and integrability of 3D Euler and Navier-Stokes equations for rotating fluids. \emph{Asymptotic Anal.} \textbf{15} (1997), 103--150.

\bibitem{BMN99} A. Babin, A. Mahalov, B. Nicolaenko. Global regularity of 3D rotating Navier-Stokes equations for resonant domains. \emph{Indiana Univ. Math. J.} \textbf{48} (1999), no. 3, 1133--1176.

\bibitem{Berselli} L. C. Berselli, T. Iliescu, W. J. Layton. Mathematics of Large Eddy Simulation of Turbulent Flows, Springer, Berlin 2005.

\bibitem{Bous} J. Boussinesq. Essai sur la th\'{e}orie des eaux courantes, M\'{e}moires pr\'{e}sent\'{e}s par divers savants \`{a} l'Acad\'{e}mie des Sciences XXIII (1877), 1--680.

\bibitem{Cummins} P. F. Cummins, G. Holloway. Reynolds stress and eddy viscosity in direct numerical simulations of sheared two-dimensional turbulence. \emph{J. Fluid Mech.} \textbf{657} (2010), 394--412.

\bibitem{DebPapp} A. Debussche, U. Pappalettera. Second order perturbation theory of two-scale systems in fluid dynamics, arXiv:2206.07775.

\bibitem{FGL21a} F. Flandoli, L. Galeati, D. Luo. Scaling limit of stochastic 2D Euler equations with transport noises to the deterministic Navier-Stokes equations. \emph{J. Evol. Equ.} \textbf{21} (2021), no. 1, 567--600.

\bibitem{FGL21b} F. Flandoli, L. Galeati, D. Luo. Delayed blow-up by transport noise. \emph{Comm. Partial Differential Equations} \textbf{46} (2021), no. 9, 1757--1788.

\bibitem{FGL21c} F. Flandoli, L. Galeati, D. Luo. Quantitative convergence rates for scaling limit of SPDEs with transport noise. arXiv:2104.01740v2.

\bibitem{FG95} F. Flandoli, D. Gatarek. Martingale and stationary solutions for stochastic Navier-Stokes equations. \emph{Probab. Theory Related Fields} \textbf{102} (1995), no. 3, 367--391.

\bibitem{FH} F. Flandoli, R. Huang. Noise based on vortex structures in 2D and 3D, arXiv:2210.12424.

\bibitem{FL} F. Flandoli, D. Luo. High mode transport noise improves vorticity blow-up control in 3D Navier-Stokes equations. \emph{Probab. Theory Related Fields} \textbf{180} (2021), no. 1--2, 309--363.

\bibitem{FLL23} F. Flandoli, D. Luo, E. Luongo. 2D Smagorinsky type large eddy models as limits of stochastic PDEs, arXiv:2302.13614.

\bibitem{FP} F. Flandoli, U. Pappalettera. From additive to transport noise in 2D fluid dynamics. \emph{Stoch. Partial Differ. Equ. Anal. Comput.} \textbf{10} (2022), no. 3, 964--1004.

\bibitem{Galeati20} L. Galeati. On the convergence of stochastic transport equations to a deterministic parabolic one. \emph{Stoch. Partial Differ. Equ. Anal. Comput.} \textbf{8} (2020), no. 4, 833--868.

\bibitem{JiangLayton} N. Jiang, W. Layton, M. McLaughlin, Y. Rong, H. Zhao. On the foundations of eddy viscosity models of turbulence. \emph{Fluids} \textbf{2020}, 5, 167.

\bibitem {KrauseRadler}F. Krause, K.-H. R\"{a}dler, Mean Field Magnetohydrodynamics and Dynamo Theory, Pergamon Press, Oxford 1980.

\bibitem{LabLay} A. Labovsky, W. Layton. Magnetohydrodynamic flows: Boussinesq conjecture. \emph{J. Math. Anal. Appl.} \textbf{434} (2016), 1665--1675.


\bibitem{Luo21a} D. Luo. Convergence of stochastic 2D inviscid Boussinesq equations with transport noise to a deterministic viscous system. \emph{Nonlinearity} \textbf{34} (2021), no. 12, 8311--8330.

\bibitem{Luo23} D. Luo. Regularization by transport noises for 3D MHD equations. \emph{Sci. China Math.} (2022). https://doi.org/10.1007/s11425-021-1981-9.

\bibitem{MarPul} C. Marchioro, M. Pulvirenti. Mathematical theory of incompressible nonviscous fluids. Applied Mathematical Sciences, 96. \emph{Springer-Verlag, New York}, 1994.

\bibitem{Proud} J. Proudman. On the motion of solids in a liquid possessing vorticity. \emph{Philos. Trans. R. Soc. London A} \textbf{92} (1916), 408--424.

\bibitem{RL18} B. L. Rozovsky, S. V. Lototsky. Stochastic evolution systems. Linear theory and applications to non-linear filtering. Probability Theory and Stochastic Modelling, 89. \emph{Springer, Cham}, 2018.

\bibitem{Schmi} F. G. Schmitt. About Boussinesq's turbulent viscosity hypothesis: historical remarks and a direct evaluation of its validity. \emph{C. R. Mecanique} \textbf{335} (2007), 617--627.

\bibitem{Seshasa} K. Seshasayanan, A. Alexakis, Turbulent 2.5 dimensional dynamos. \emph{J. Fluid Mech.} \textbf{799} (2016), 246--264.

\bibitem{Taylor} G. I. Taylor. Motion of solids in fluids when the motion is not irrotational. \emph{Proc. R. Soc. London A} \textbf{93} (1917), 99--113.

\bibitem{Wirth} A. Wirth, S. Gama, U. Frisch. Eddy viscosity of three-dimensional flows. \emph{J. Fluid Mech.} \textbf{288} (1995), 249--264.

\end{thebibliography}
\end{document}